\documentclass[reqno]{amsart}
\usepackage{amsmath,amsbsy,amsfonts}
\usepackage{color}
\usepackage{epsfig}

\date{}
\def\nd{\noindent}
\def\thend{\rule{3mm}{3mm}}

\newtheorem{theorem}{Theorem}[section]

\newtheorem{prop}{Proposition}[section]
\newtheorem{lem}{Lemma}[section]
\newtheorem{rmk}{Remark}[section]

\newcommand{\Int}{\displaystyle\int_{\Rn}}

\newcommand{\2}{2^*}

\newcommand{\Rn}{\mathbb{R}^N}

\begin{document}
\title[Choquard equations via nonlinear Rayleigh quotient] {Choquard equations via nonlinear Rayleigh quotient for concave-convex nonlinearities}
\vspace{1cm}

\author{M. L. M. Carvalho}
\address{M. L. M. Carvalho \newline Universidade Federal de Goias, IME, Goi\^ania-GO, Brazil }
\email{\tt marcos$\_$leandro$\_$carvalho@ufg.br}

\author{Edcarlos D. da Silva}
\address{Edcarlos D da Silva \newline  Universidade Federal de Goias, IME, Goi\^ania-GO, Brazil}
\email{\tt edcarlos@ufg.br}

\author{C. Goulart}
\address{C. Goulart \newline Universidade Federal de Jata\'\i, Jata\'\i-GO, Brazil }
\email{\tt claudiney@ufg.br}


\thanks{The second author was partially supported by CNPq/Universal 2018 with grant 429955/2018-9}

\begin{abstract}
	It is established existence of ground and bound state solutions for Choquard equation considering concave-convex nonlinearities in the following form
	\begin{equation*}
	\left\{
	\begin{array}{rcl}
	-\Delta u  +V(x) u &=&  (I_\alpha* |u|^p)|u|^{p-2}u+ \lambda |u|^{q-2}u \, \mbox{ in }\,    \mathbb{R}^N,  \\
	u\in H^1(\Rn)&&
	\end{array}
	\right.
	\end{equation*}
	where $\lambda > 0, N \geq 3, \alpha \in (0, N)$. The potential $V$ is a continuous function and $I_\alpha$ denotes the standard Riesz potential. Assume also that $1 < q < 2$, $2_\alpha < p < 2^*_\alpha$ where $2_\alpha=(N+\alpha)/N$, $ \2_\alpha=(N+\alpha)/(N-2)$. Our main contribution is to consider a specific condition on the parameter $\lambda > 0$ taking into account the nonlinear Rayleigh quotient. More precisely, there exists $\lambda_n > 0$ such that our main problem admits at least two positive solutions for each $\lambda \in (0, \lambda_n]$. In order to do that we combine Nehari method with a fine analysis on the nonlinear Rayleigh quotient. The parameter $\lambda_n > 0$ is optimal in some sense which allow us to apply the Nehari method.  
\end{abstract}

\maketitle

\section{Introduction}

It is well known that existence, nonexistence and multiplicity of solutions for nonlocal elliptic problems are related with the behavior for nonlinear term at the origin and at infinity. In this work we shall consider semilinear elliptic problems driven by the Choquard equation described in the following form:
\begin{equation}\label{eq1}
\left\{
\begin{array}{rcl}
-\Delta u  +V(x) u &=&  (I_\alpha* |u|^p)|u|^{p-2}u+ \lambda |u|^{q-2}u \, \mbox{ in }\,    \mathbb{R}^N,  \\
u\in H^1(\Rn)&&
\end{array}
\right.
\end{equation}
where $\lambda > 0, N \geq 3, \alpha \in (0, N)$. The potential $V$ is a continuous function and $I_\alpha$ denotes the standard Riesz potential. Assume also that $1 < q < 2$, $2_\alpha < p < 2^*_\alpha$ where $2_\alpha=(N+\alpha)/N$, $ \2_\alpha=(N+\alpha)/(N-2)$. Later on, we shall consider hypotheses on $V$ and $\lambda$. 
Recall that the Riesz potential can be described in the following form
$$I_\alpha(x)=\dfrac{A_\alpha(N)}{|x|^{N-\alpha}}, x \in \mathbb{R}^N \,\, \mbox{and} \, \,A_\alpha(N)=\dfrac{\Gamma(\frac{N-\alpha}{2})}{\Gamma(\frac{\alpha}{2})\pi^{\frac{N}{2}}2^s},  $$
where $\Gamma$ denotes the Gamma function,  see \cite{Muroz0}. The Choquard equation has many physical applications. For example assuming that $N = 3, \alpha = 2, p = 2, \lambda = 0$ and
$V \equiv 0$, Problem \eqref{eq1} was investigated in \cite{Pekar} to study the quantum theory of a polaron at rest. It was
pointed in \cite{Lieb} that Choquard applied it as an approximation to Hartree–Fock theory of one component
plasma. It also arises in multiple particles systems \cite{Gross} and quantum mechanics \cite{Penrose}. Furthermore, 
for each solution $u$ of Problem \eqref{eq1} we obtain the wave-function $\Phi : \mathbb{R}\times \mathbb{R}^N \to \mathbb{C}$ defined by $\Phi(t,x)= e^{i t} u(x)$ where $i$ is the imaginary unit. Hence $\Phi$ is a solitary wave of the focusing time-dependent Hartree equation $$ - i\Phi_t -  \Delta \Phi + W(x) \Phi - (I_\alpha  * |\Phi|^p) |\Phi|^{p-2}\Phi - \lambda |\Phi|^{q-2}\Phi = 0, \,\, \mbox{in} \,\,\mathbb{R} \times \mathbb{R}^N$$ where $W(x) = V(x) - 1, x \in \mathbb{R}^N$. Hence Problem \eqref{eq1} can be understood as the stationary nonlinear Hartree equation.

It is important to emphasize that nonlocal elliptic problems involving Choquard equations have been studied in the last years taking into account several kinds of assumptions on the potential $V$. Here we refer the interested reader to the works \cite{Muroz0,Muroz1, Muroz2,Li} and references therein. In these works was considered existence, nonexistence and quality properties of weak solutions for Choquard equations assuming that $\lambda = 0$. In other words, semilinear elliptic problems involving the Choquard equation have been widely considered assuming that the nonlinearity is superlinear at infinity and at the origin. For further results on nonlocal elliptic problems involving the Choquard equation we refer to \cite{Chen, Li}.

Our main contribution is to consider existence and multiplicity of solutions for the Problem \eqref{eq1} where the nonlinearity is concave-convex. This kind of problem for the local case have been extensively considered in the last years. Here we cite the pioneer work \cite{ABC} where several results are proved on bounded domains $\Omega \subset \mathbb{R}^N$. In the whole space concave-convex nonlinearities have been considered assuming extra assumptions on the potential $V$, see \cite{Wu1, Wu2,Wu0}. Another contribution in this work is to consider the nonlinear Rayleigh quotient proving existence of a parameter $\lambda_n > 0$ such that Problem \eqref{eq1} admits at least two solutions for each $\lambda \in (0, \lambda_n]$. The main point here is to ensure that the Nehari method can be applied for each $\lambda \in (0, \lambda_n)$. In fact, to the best our knowledge, this is the first work considering the Choquard equation using the nonlinear Rayleigh quotient together with a fine analysis on the Nehari method.  Furthermore, by using a sequence procedure, we consider the behavior for these solutions when the parameter $\lambda$ goes to zero or $\lambda_n$. Finally, we also consider a regularity result for our main problem. More specifically, we show that any weak solution for the Problem \eqref{eq1} is in $C_{loc}^{1,\beta}(\mathbb{R}^N)$ for some $\beta \in (0,1)$, see Appendix ahead. Notice also that for $\lambda=\lambda_n$ the Choquard term brings us some difficulties. The first one is to ensure that minimizers on the Nehari manifold yields a critical point for the energy functional. The second one is to guarantee that any minimizer in the Nehari manifold is small in the set $B_R(0)^c = \{x \in \mathbb{R}^N, |x| > R \}$ for some $R > 0$ large enough, that is, for each minimizer $u \in X$ in the Nehari manifold we need to show that $|u(x)| \to 0$ as $|x| \to \infty$. In order to overcome these difficulties we prove a regularity result together with an appropriate behavior for the Choquard term, see Appendix. Hence we can prove our main results assuming that $\lambda \in (0, \lambda_n)$ finding existence of two positive solutions for Problem \eqref{eq1} for each $\lambda \in (0, \lambda_n]$.

It is important to worthwhile that nonlinear Rayleigh quotient have been studied in the last years, see \cite{marcos,yavdat0, yavdat1, yavdat2}. The main feature in these works is to guarantee that there exists an extreme value $\lambda_n > 0$ in such way that the Nehari method can be applied for each $\lambda  \in (0, \lambda_n)$. The basic idea is to ensure that the fibering map admits at least two critical points for each $\lambda \in (0, \lambda_n)$. The same can be done for our main problem taking into account the convolution term which bring us some difficulties. The first one is to control the behavior at infinity and at the origin for the fibering maps which is the key point for our arguments. The second difficulty arises from in order to classify the signal for energy functional. More specifically, we obtain existence of a positive ground state solution $u_\lambda$ such that $E_\lambda(u_\lambda) < 0$ for each $\lambda \in (0, \lambda_n]$ where $E_\lambda$ is the energy functional for our main problem. For the second positive solution $v_\lambda$ the problem is more involved. For this solution we need to consider another Rayleigh quotient in order to get extra informations for the signal of $E_\lambda(v_\lambda)$. This can be done proving existence of a parameter $\lambda_e > 0$ such that $E_\lambda(v_\lambda) > 0$ for each $\lambda \in (0, \lambda_e)$. In the same way, assuming that $\lambda = \lambda_e$ we obtain $E_\lambda(v_\lambda) = 0$, i.e, $v_\lambda$ is a positive solution with zero energy. For the case $\lambda \in (\lambda_e, \lambda_n)$ we obtain existence of two positive solutions with negative energy. Furthermore, we prove the same result using a sequence whenever $\lambda = \lambda_n$. More specifically, given any sequence $(\lambda_j)$ such that $\lambda_j \to \lambda_n$ as $j \to \infty$ with $\lambda_j < \lambda_n$ we ensure that Problem \eqref{eq1} admits at least weak solutions for $\lambda = \lambda_n$. Moreover, for each $i = 1, 2$ we consider the continuity for the function $\lambda\mapsto\mathcal{E}_\lambda^i$ where $\mathcal{E}_\lambda^i$ denotes the minimal energy level on the Nehari manifolds $\mathcal{N}_\lambda^-$ and $\mathcal{N}_\lambda^+$, see Theorems \ref{theorem2} and \ref{theorem3} ahead. Hence our main results complement the aforementioned works.

\subsection{Assumptions and main theorems}
As mentioned in the introduction, we are concerned with the existence of ground and bound states for Problem ~\eqref{eq1} involving concave-convex nonlinearities. In this case, we need to control the parameter $\lambda > 0$ getting our main results. In order to overcome this difficulty, we shall consider the nonlinear Rayleigh quotient showing that there exists $\lambda_n > 0$ such that the Nehari method can be applied for each $\lambda \in (0, \lambda_n]$. More specifically, we study the existence of ground and bound state solutions for our main problem taking advantage that the nonlinearity is concave-convex.

Throughout this work we assume the following assumptions:
\begin{itemize}
	\item[$(Q)$] It holds $1 < q < 2$ and  $p\in(2_\alpha,\2_\alpha)$ with $2_\alpha= (N+\alpha)/N$, $ \2_\alpha= (N+\alpha)/(N-2)$;
	\item [$(V_1)$] The function $V:\mathbb{R}^N\to \mathbb{R}$ is continuous and there exists a constant $V_0>0$ such that $$V(x)\geq V_0 \,\, \mbox{for all} \,\, x \in \mathbb{R}^N;$$
	\item [$(V_2)$] It holds $V^{-1}\in L^1(\mathbb{R^N})$, i.e., the function $V$ satisfies the following integrability condition $$\Int \dfrac{1}{V(x)}dx<+\infty.$$
\end{itemize}

Now we consider the working space for our problem defined by
$$X= \left\{v\in H^1(\mathbb{R}^N): \Int V(x)v^2dx<+\infty \right\}.$$ 
Notice that $X$ is a Banach space which is endowed with the norm
\begin{equation}
||v||:=\left(\Int (|\nabla u|^2 + V(x)u^2)dx\right)^{\frac12}.
\end{equation}
Under our hypotheses, the embedding $X\hookrightarrow L^r(\mathbb{R}^N)$ is continuous for each $r\in [1,2^*]$. Furthermore, the same embedding is compact for each $r \in [1,2^*)$, 
see for instance \cite{citub}. It is worthwhile to mention that the energy functional associated to Problem \eqref{eq1} is given by 
\begin{equation}\label{functional}
E_{\lambda}(u)=\dfrac{1}{2}||u||^2-\dfrac{1}{2p}\Int\left(I_\alpha*|u|^p\right)|u|^pdx-\dfrac{\lambda}{q}\Int |u|^qdx, \, u\in X.
\end{equation}
Now we define the inner product in $X$ as follows
\begin{equation}\label{normas}
\left<u,\psi\right>=\Int\left(\nabla u\nabla \psi+V(x) u\psi\right)dx, \, u, \psi \in X.
\end{equation}
It is worthwhile to mention that a function $u \in X$ is said to be a weak solution for Problem \eqref{eq1} whenever 
\begin{equation}
\left<u,\psi\right>-\Int\left(I_\alpha*|u|^p\right)|u|^{p-2}u\psi dx-\lambda\Int|u|^{q-2}u\psi dx = 0 \,\, \mbox{for every} \,\, \psi\in H^1(\mathbb{R}^N).
\end{equation}
Using the embedding $X\hookrightarrow L^r(\mathbb{R}^N)$ for each $r\in [1,2^*]$ it is well known that $E_\lambda \in C^1(X,\mathbb{R})$. Furthermore, the Gateaux derivative for $E_\lambda$ is given by 
\begin{equation}
E'_\lambda(u)\psi=\left<u,\psi\right>-\Int\left(I_\alpha*|u|^p\right)|u|^{p-2}u\psi dx-\lambda\Int|u|^{q-2}u\psi dx, \,\, \mbox{for every} \,\, u,\psi\in X.
\end{equation}
Hence, a function $u \in X$ is a weak solution to the elliptic Problem \eqref{eq1} if and only if $u$ is a critical point for the functional $E_\lambda$. Notice also that we consider the last identity for any testing function $\psi \in X$. However, using some estimates we can use any test function $\psi \in H^1(\mathbb{R}^N)$, see Proposition \ref{ahead} ahead. In this way, we can apply variational methods in order to ensure that the functional $E_\lambda$ admits critical points. It is important to recall that a nontrivial solution $u\in X$ is called a ground state solution for Problem \eqref{eq1} provided that $u$ satisfies the following identity
\begin{equation}\label{ground}
E_\lambda(u)=\inf\{E_\lambda(v): v\in X\setminus\{0\} \mbox{ and } E'_\lambda(u)=0 \}.
\end{equation}

It is important to emphasize that $\Lambda_e,\Lambda_n: X \setminus \{0\} \rightarrow \mathbb{R}$ are $C^1$ functions given by
\begin{equation}
\Lambda_e(u) = \max_{t>0}\dfrac{\frac{t^{2-q}}{2}||u||^2-\frac{t^{2p-q}}{2p}\int_{\Rn}\left(I_\alpha*|u|^p\right)|u|^pdx}{\frac{1}{q}\int_{\Rn}|u|^qdx}
\end{equation}
and
\begin{equation}
\Lambda_n(u) =\max_{t>0}\dfrac{t^{2-q}||u||^2-t^{2p-q}\int_{\Rn}\left(I_\alpha*|u|^p\right)|u|^pdx}{\int_{\Rn}|u|^qdx}.
\end{equation}
At this stage, we shall consider the nonlinear Rayleigh quotient as follows
\begin{eqnarray}
\lambda_e:=\inf_{u\in X\setminus\{0\}} \Lambda_e(u) \label{lambdae}
\end{eqnarray}
and
\begin{eqnarray}
\lambda_n:=\inf_{u\in X\setminus\{0\}} \Lambda_n(u).\label{lambdan}
\end{eqnarray}

In our main results we shall consider the Nehari set as follows
\begin{equation}
\begin{array}{rcl}
\mathcal{N_\lambda}&=&\left\{ u\in  X\setminus \{0\}:||u||^2-\Int\left(I_\alpha*|u|^p\right)|u|^pdx=\lambda{\Int|u|^qdx}\right\}, \lambda > 0.
\end{array}
\end{equation}
For the Nehari method we refer the reader to \cite{nehari1,nehari2}. The Nehari set can be separated in the following form:
\begin{eqnarray}
\mathcal{N_\lambda}^+&=&\{u\in\mathcal{N}_\lambda: E''_\lambda(u)(u,u)>0 \}\label{n+}, \nonumber \\ 
\mathcal{N_\lambda}^-&=&\{u\in\mathcal{N}_\lambda: E''_\lambda(u)(u,u)<0 \}\label{n-}, \nonumber \\
\mathcal{N_\lambda}^0&=&\{u\in\mathcal{N}_\lambda: E''_\lambda(u)(u,u)=0\}\label{n0}. \nonumber 
\end{eqnarray}
It is not hard to verify that the function $u \mapsto E''_\lambda(u)(u,u)$ is well defined for each $u \in X$. Namely, we can deduce the following expression 
\begin{equation}\label{segunda}
E''_\lambda(u)(u,u) = ||u||^2- (2 p -1)\Int\left(I_\alpha*|u|^p\right)|u|^pdx-\lambda (q - 1)\Int |u|^qdx, \, u\in X.
\end{equation}
The main objective here is to find solutions for the following minimization problems
\begin{equation}\label{ee1}
\mathcal{E}_\lambda^1:=\inf\{E_\lambda(u): u\in\mathcal{N_\lambda}^+\}
\end{equation}
\begin{equation}\label{ee2}
\mathcal{E}_\lambda^2:=\inf\{E_\lambda(u): u\in\mathcal{N_\lambda}^-\}.
\end{equation}
It is not hard to verify that any minimizer $u_\lambda \in \mathcal{N}_\lambda^+$ give us a ground state solution for the Problem \eqref{eq1}. This can be done comparing the energy levels given by \eqref{ground} and  \eqref{ee1}.
In order to find minimizer in $\mathcal{N}_\lambda^+$ and $\mathcal{N}_\lambda^-$ we need to consider some extra assumptions. Indeed, assuming that $\lambda \in (0, \lambda_n)$, we shall prove that $\mathcal{E}_\lambda^1$ and $\mathcal{E}_\lambda^1$ are attained. It is important to mention that $\mathcal{N}_{\lambda}^0$ is empty for each $\lambda \in (0, \lambda_n)$. This fact allows us to apply the Nehari method taking into account the uniqueness of the projections in $\mathcal{N}_\lambda^+$ and $\mathcal{N}_\lambda^-$. This is the main feature for the parameter $\lambda_n > 0$. In fact, the parameter $\lambda_n$ is the first positive number in such way that $\mathcal{N}_\lambda^0$ is not empty, see \cite{yavdat1}. In this way, we can state our first main result in the following way:
\begin{theorem}\label{theorm1}
	Suppose $(Q)$ and $(V_1)-(V_2)$. Then $0<\lambda_e<\lambda_n<\infty$ and for each $\lambda\in(0,\lambda_n)$ the Problem \eqref{eq1} admits at least two distinct positive solutions $u_\lambda, v_\lambda \in X$ satisfying the following statements:  $E''_\lambda(u_\lambda)(u_\lambda,u_\lambda)>0,\  E''_\lambda(v_\lambda)(v_\lambda,v_\lambda)<0, \, \ E_\lambda(u_\lambda)<0$ and  $u_\lambda \in \mathcal{N}_\lambda^+, v_\lambda \in \mathcal{N}_\lambda^-$.  Furthermore, $u_\lambda$ is a ground state solution and $v_\lambda$ satisfies the following statements: 
	\begin{itemize}
		\item[(i)] For each $\lambda\in (0,\lambda_e)$ we obtain that $E_\lambda(v_\lambda)>0$;
		\item[(ii)] For each $\lambda=\lambda_e$ we deduce that $E_\lambda(v_\lambda)=0$;
		\item[(iii)] For each $\lambda\in (\lambda_e, \lambda_n)$ we obtain also that $E_\lambda(v_\lambda)<0$.
	\end{itemize}
\end{theorem}
\begin{figure}[!ht]
	\begin{minipage}[h]{0.49\linewidth}
		\center{\includegraphics[scale=0.7]{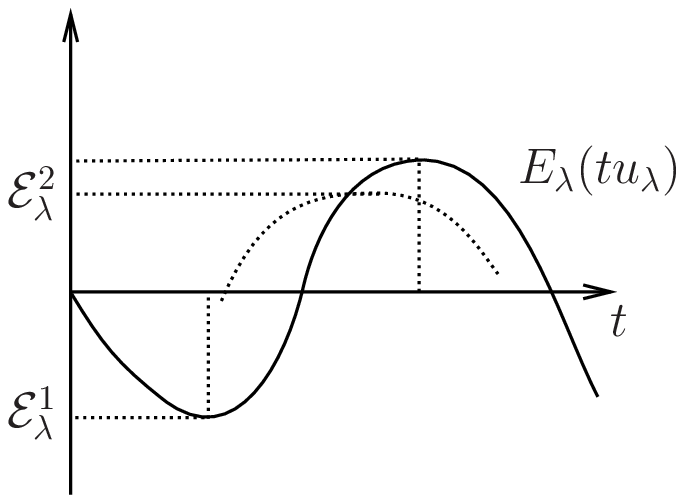}}
		\caption{$\lambda\in (0,\lambda_e)$}
		\label{figE}
	\end{minipage}
	\begin{minipage}[h]{0.49\linewidth}
		\center{\includegraphics[scale=0.7]{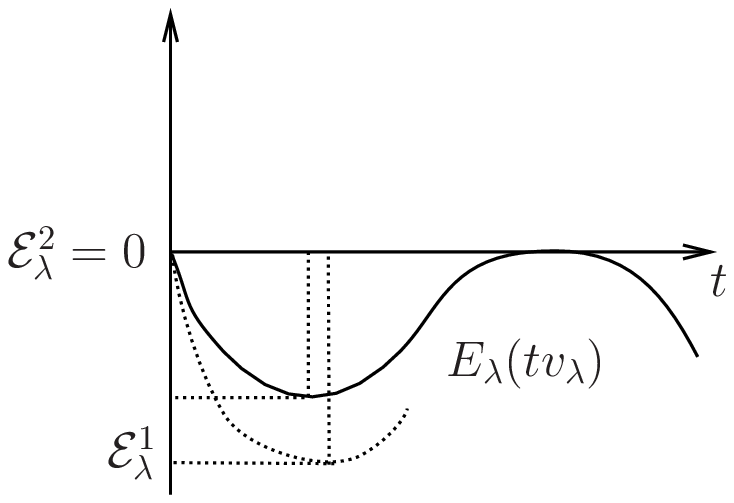}}
		\caption{$\lambda=\lambda_e$ }
		\label{figE1}
	\end{minipage}
	\begin{minipage}[h]{0.49\linewidth}
		\center{\includegraphics[scale=0.7]{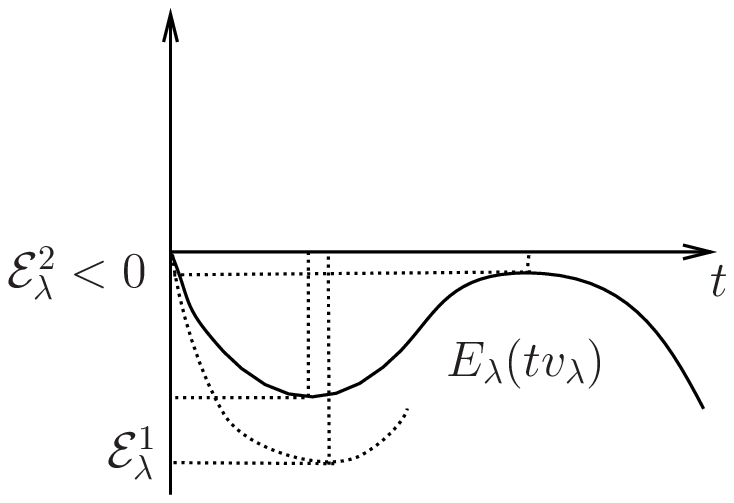}}
		\caption{$\lambda\in(\lambda_e,\lambda_n)$ }
		\label{figE2}
	\end{minipage}
\end{figure}

\begin{rmk}
	Under assumptions of Theorem \ref{theorm1} we obtain two positive solutions $u_\lambda$ and $v_\lambda$ for each $\lambda \in (0, \lambda_n)$. More specifically, $u_\lambda$  and $v_\lambda$ are solutions for the minimizations problems \eqref{ee1} and \eqref{ee2}, respectively. Furthermore, we classify the signal of $E_\lambda(v_\lambda)$ depending on the size of $\lambda$ where $\lambda \in (0, \lambda_n)$, see Figures \ref{figE}, \ref{figE1}, \ref{figE2}. 
\end{rmk}
 
It follows from Theorem \ref{theorm1} that $\mathcal{E}^1_{\lambda}:=E_\lambda(u_\lambda)$ and $\mathcal{E}^2_{\lambda}:=E_\lambda(v_\lambda)$. In other words, $\mathcal{E}^1_{\lambda}$ and $\mathcal{E}^2_{\lambda}$ are attained for each $\lambda \in (0, \lambda_n)$. 
Under these conditions we are able to state our second main result in the following form:
\begin{theorem}\label{theorem2}
Suppose $(Q)$ and $(V_1)-(V_2)$.	Let $\widetilde{\lambda}\in(0,\lambda_n)$ be fixed. Assume that $(\lambda_j)\subset (0,\lambda_n)$ such that  $\lambda_j\to\widetilde\lambda$. Then we obtain that the following assertions:
	\begin{itemize}
		\item[i)] The functions $\lambda\mapsto\mathcal{E}_\lambda^1$ and $\lambda\mapsto\mathcal{E}_\lambda^2$ are decreasing and $\mathcal{E}_\lambda^1<\mathcal{E}_\lambda^2$;
		\item[ii)] $u_{\lambda_j} \rightarrow u_{\widetilde\lambda}$ and $v_{\lambda_j} \rightarrow v_{\widetilde\lambda}$ in $X$ as $j \rightarrow \infty$;
		\item[iii)] $\mathcal{E}^i_{\lambda_j} \to \mathcal{E}^i_{\widetilde\lambda}$ as
		$j\to \infty$, that is, $\lambda \mapsto \mathcal{E}^i_{\lambda}$ are continuous functions for each $i=1,2$.
	\end{itemize}
\end{theorem}

\begin{rmk}
	Under hypotheses of Theorem \ref{theorem2} it follows that $\lambda\mapsto\mathcal{E}_\lambda^1$ and $\lambda\mapsto\mathcal{E}_\lambda^2$ are decreasing continuous functions for each $\lambda \in (0, \lambda_n)$. These facts imply that $\mathcal{E}_{\tilde\lambda}^1$ and $\mathcal{E}_{\tilde\lambda}^2$ are close to $\mathcal{E}_\lambda^1$ and $\mathcal{E}_\lambda^2$ for $\tilde \lambda$ being next to $\lambda$, respectively.
\end{rmk}

In the next result we shall consider the behavior of $\mathcal{E}_\lambda^1$ and $\mathcal{E}_\lambda^2$ when $\lambda \to 0$. For this case, the concave term disappear proving that $u_\lambda$ goes to zero as $\lambda \to 0$. More precisely, we can state the following result: 
\begin{theorem}\label{theorem3}
	Suppose $(Q)$ and $(V_1)-(V_2)$. Assume that $(\lambda_j)\subset (0,\lambda_n)$ such that  $\lambda_j\to 0$. Then we obtain that the following assertions:
	\begin{itemize}
		\item[i)] $\mathcal{E}^1_{\lambda_j} \to 0$ and $ \mathcal{E}^2_{\lambda_j} \to \mathcal{E}^2_{0}$ as $j\to \infty$, that is, the function $\lambda\mapsto \mathcal{E}^i_{\lambda}$ is right continuous at $\lambda=0$ for each $i=1,2$.
		\item[ii)] $u_{\lambda_j} \rightarrow 0$ and $v_{\lambda_j} \rightarrow v_0$ in $X$ as $j \rightarrow \infty$ where $v_0$ is a positive solution of Problem \eqref{eq1} with $\lambda=0$.
	\end{itemize}
\end{theorem}

For the next result we shall consider the case $\lambda= \lambda_n$. For this case we mention that $\mathcal{N}_{\lambda_n}^0 \neq \emptyset$. The parameter $\lambda_n$ is the smallest positive number in such way that $\mathcal{N}_{\lambda}^{0}$ is a nonempty set, see \cite{yavdat1}. The last assertion implies that the Nehari method can not be applied directly. Here we need to control the functions $u_\lambda$ and $v_\lambda$ which are the minimizers for $\mathcal{E}_\lambda^1$ and $\mathcal{E}_\lambda^2$. More specifically, we need to ensure that $u_\lambda$ and $v_\lambda$ does not belong to $\mathcal{N}_{\lambda}^{0}$. This can be done using a regularity result together with an asymptotic $u_\lambda$ and $v_\lambda$ at infinity. Due to the nonlocal term we consider some fine estimates in order to show that any weak solution for our main problems is smooth proving that any critical point for the functional $E_\lambda$ does not belong to $\mathcal{N}_{\lambda_n}^0$, see Theorem \ref{regular} and Proposition \ref{salva} ahead. In this way, we can ensure the following result:  

\begin{theorem}\label{theorem4}
	Suppose $(Q)$ and $(V_1)-(V_2)$. Assume also that $\lambda = \lambda_n$. Then Problem \eqref{eq1} admits at least two positive solutions $u_{\lambda_n}$ and $v_{\lambda_n}$. Furthermore, $\mathcal{E}_{\lambda_n}^1 < \mathcal{E}_{\lambda_n}^2$.
\end{theorem}

\begin{rmk}
Under hypotheses of Theorem \eqref{theorem4} it follows that $u_{\lambda_n}$ and $v_{\lambda_n}$ have negative energy. However, $u_{\lambda_n} \in \mathcal{N}_{\lambda_n}^+$ and $v_{\lambda_n} \in \mathcal{N}_{\lambda_n}^-$ are the minimizer for $\mathcal{E}_\lambda^1$ and $\mathcal{E}_\lambda^2$, respectively. The last assertion implies that $u_{\lambda_n}$ and $u_{\lambda_n}$ are distinct solutions.	
\end{rmk}

\subsection{Notation} Throughout this work we shall use the following notation:

\begin{itemize}
	\item $C$, $\tilde{C}$, $C_{1}$, $C_{2}$,... denote positive constants (possibly different).
	\item The norm in $L^{p}(\mathbb{R}^N)$ and $L^{\infty}(\mathbb{R}^N)$, will be denoted respectively by $\|\cdot\|_{p}, p \in [1, \infty),$ and $\|\cdot\|_{\infty}$.
	\item The norm in $X$ is denoted by $\|u\|^2= \Int (|\nabla u|^2 + V(x)u^2)dx, u \in X$.
\end{itemize}


\subsection{Outline} The remainder of this work is organized as follows: In the forthcoming section we consider some preliminary results together with the Nehari method for our main problem. In Section 3 we prove Theorem \ref{theorm1}. Section 4 is devoted to the proof of Theorems \ref{theorem2}, \ref{theorem3} and \ref{theorem4}. In an Appendix we consider further results for nonlocal elliptic problems involving the Choquard equation.

\section{Preliminaries}
In this section we recall some basic properties together with the Nehari method for our main problem. Firstly, we consider some powerful tools for nonlocal elliptic problems involving the Choquard equations. It is well known the Hardy-Littlewood-Sobolev inequality \cite{Muroz0} which can be stated as follows: 
\begin{lem}\label{lemaHardy} 
	Let $t,r>1$ and $ 0<\alpha<N$ such that $\frac{1}{t}+\frac{N-\alpha}{N}+\frac{1}{r}=2$. Let $\phi\in L^{t}(\mathbb{R}^N)$ and $\psi\in L^s(\Rn)$ be fixed functions. Then there exists a sharp constant $C = C(N, \alpha, t)>0$ such that 
	$$\Int\Int \dfrac{\left|\phi(x) \psi(y)\right|}{|x-y|^{N-\alpha}} dx dy\leq C||\phi||_t||\psi||_s.$$ 
	Furthermore, assuming that $t=r=\dfrac{2N}{N+\alpha}$, we mention that $$C=\pi^{N-\alpha}\dfrac{\Gamma(\frac{\alpha}{2})}{\Gamma(\frac{N+2}{2})}\left\{\dfrac{\Gamma(\frac{N}{2})}{\Gamma(N)}\right\}^{\frac{-\alpha}{N}}.$$
\end{lem}
\begin{rmk} \label{rmHardy} It follows from Lemma \ref{lemaHardy} that $I_{\alpha} * \psi \in L^{\frac{Ns}{N-\alpha s}}(\Rn)$ holds for all $\psi \in L^{s}(\mathbb{R}^N)$. Furthermore, we know that
	$$	||I_{\alpha}*\psi||_{\frac{Ns}{N-\alpha s}}\leq A_\alpha(N)C(N,\alpha,s)||\psi||_s$$
	holds for each $s\in(1,\frac{N}{\alpha})$ where $A_{\alpha}(N) > 0$, see for instance \cite{Muroz0}. 
\end{rmk}
	For the next result we shall prove that any critical point $u \in X$ for the energy functional $E_\lambda$ is a weak solution for the Problem \eqref{eq1}. More precisely, we consider the following result
	\begin{prop}\label{ahead}
		Suppose $(Q)$ and $(V_1)-(V_2)$. Assume that $u \in X$ is a critical point for the functional $E_\lambda$, that is, $E'_\lambda (u) \psi = 0$ for any $\psi \in X$. Then we obtain 
		\begin{equation*}
		\int_{ \mathbb{R}^{N}}\left(\nabla u \nabla \psi + V(x) u \psi\right)dx=\int_{ \mathbb{R}^{N}} g(u)\psi dx
		\end{equation*} 
		holds for any $\psi \in H^1(\mathbb{R}^N)$ where $g(u) = (I_\alpha* |u|^p)|u|^{p-2}u + \lambda |u|^{q -2}u$.
	\end{prop}
	\begin{proof}
		Let $\psi \in H^1(\mathbb{R}^N)$ be a fixed nonnegative function. Consider $(\eta_{k})_{k\in\mathbb{N}}$ a sequence of functions in $C^{\infty}(\mathbb{R}^N)$ such that
		\begin{equation*}
		\begin{gathered}
		\begin{cases}
		0\leq \eta_{k}(x)\leq 1\quad\mbox{for all}\quad x\in\mathbb{R}^N,\\
		\eta_{k}(x)=1\quad\mbox{if}\quad |x|\leq k,\\
		\eta_{k}(x)=0\quad\mbox{if}\quad |x|\geq k+1,\\
		\left\lvert\nabla\eta_{k}(x)\right\rvert\leq C\quad\mbox{for all}\quad x\in\mathbb{R}^N.
		\end{cases}
		\end{gathered}
		\end{equation*}
		It is easy to verify that $\eta_{k}(x)\psi\in H^1(\mathbb{R}^N)$, $\eta_{k}(x)\psi(x)\to\psi(x)$ a.e. $x\in\mathbb{R}^N$. Furthermore, we mention that $\eta_{k}(x)\psi\to\psi$ in $H^1(\mathbb{R}^N)$. Hence, by using the continuous Sobolev embedding $H^1(\mathbb{R}^N)\hookrightarrow L^r(\mathbb{R}^N), r \in [2, 2^*], 2^* = 2 N/(N - 2)$, we obtain $\eta_{k}(x)\psi\to\psi$ in $L^s(\mathbb{R}^N)$ for each $2\leq s\leq 2^*$. Moreover, by using the fact that $V$ is locally bounded and $\eta_{k}(x)\psi$ has compact support, we infer that $\eta_{k}(x)\psi\in X$ holds for each $k\in\mathbb{N}$. As a product 
		\begin{equation}\label{q}
		\int_{ \mathbb{R}^{N}}\nabla u\nabla \left(\eta_{k}(x)\psi\right)+V(x) u ( \eta_{k}(x)\psi ) dx=\int_{ \mathbb{R}^{N}}g(u)\eta_{k}(x)\psi dx.
		\end{equation} 
		Using the strong convergence in $H^1(\mathbb{R}^N)$ and taking into account that $g$ is subcritical we deduce that 
		\begin{eqnarray}\label{q1}
		\lim_{k\to\infty}\int_{ \mathbb{R}^{N}}\nabla u \nabla \left(\eta_{k}(x)\psi\right) dx&=&\int_{ \mathbb{R}^{N}}\nabla u\nabla \psi dx, \, \,
\lim_{k\to\infty}\int_{ \mathbb{R}^{N}}g(u)\eta_{k}(x)\psi dx =\int_{ \mathbb{R}^{N}}g(u)\psi dx\nonumber\\
		\end{eqnarray}
		On the other hand, we observe that 
		\begin{eqnarray*}
			0\geq V(x)u \eta_{k}(x)\psi(x)\geq V(x) u \eta_{k+1}(x)\psi(x)\quad \mbox{a.e}\quad x\in[u\leq0]\nonumber\\
			0\leq V(x) u \eta_{k}(x)\psi(x)\leq V(x)u \eta_{k+1}(x)\psi(x)\quad \mbox{a.e}\quad x\in[u\geq0]\nonumber 
		\end{eqnarray*}
		holds for all $k\in\mathbb{N}$. Now, by using the last assertion together with the Monotone Convergence Theorem, we get 
		\begin{equation}\label{q2}
		\lim_{k\to\infty}\int_{ \mathbb{R}^{N}}V(x) u \eta_{k}(x)\psi dx=\int_{ \mathbb{R}^{N}}V(x) u \psi dx.
		\end{equation}
		Taking the limit in \eqref{q} and using the convergences given in \eqref{q1} and \eqref{q2}, we see that 
		\begin{equation*}
		\int_{ \mathbb{R}^{N}}\left(\nabla u \nabla \psi+V(x) u \psi\right)dx=\int_{ \mathbb{R}^{N}}g(u) \psi dx
		\end{equation*} 
		holds true for each $\psi\in H^1(\mathbb{R}^N)$ satisfying $\psi\geq 0$ in $\mathbb{R}^N$. Analogously, we prove the same identity for each $\psi \in H^1(\mathbb{R}^N)$ such that $\psi \leq0$ in $\mathbb{R}^N$. The proof for the general case $\psi\in H^1(\mathbb{R}^N)$ follows immediately writing $\psi=\psi^+ + \psi^-$ where $\psi^+= \max(\psi, 0)$ and $\psi^-= \min(\psi, 0)$. This ends the proof.
	\end{proof}

Now we shall consider the Nehari method for our main Problem \eqref{eq1}. Initially, our main objective is to analyze the geometry for the fibering function $\phi : [0, \infty) \rightarrow \mathbb{R}$ given by $\phi (t)=E_\lambda(tu), t \geq 0, u \in X \setminus \{0\}$. This function was introduced in \cite{Pohozaev}. Here we refer the interested also to \cite{brow0,brow1,Pokhozhaev}. It is important to recall that the fibering map 
  is related with the Nehari manifold $\mathcal{N}_\lambda$ which was defined in the introduction in the following form:  
  \begin{equation}
  \begin{array}{rcl}
  \mathcal{N_\lambda}&:=&\{ u \in X\setminus \{0\} : \phi'(1)=E_\lambda'(u)u=0 \}\\[3ex]
  &=&\left\{ u\in  X\setminus \{0\}:||u||^2-\Int\left(I_\alpha*|u|^p\right)|u|^pdx=\lambda{\Int|u|^qdx}\right\}
  \end{array}
  \label{nehari}.
  \end{equation}
 More specifically, the critical point for the fibering map $\phi$ provide us an element on the Nehari manifold. In fact, for each $u \in X \setminus\{0\}$, there exists an unique $t > 0$ such that $t u \in \mathcal{N}_\lambda$ provided that $E''_\lambda(tu)(tu,tu) < 0$ for each $\lambda \in (0, \lambda_n)$. The same property remains true for each $u \in X \setminus \{0\}$ such that $E''_\lambda(tu)(tu,tu) > 0$ for each $\lambda \in (0, \lambda_n)$. Furthermore, any critical point $u\in X\setminus\{0\}$ for the energy functional $E_\lambda$ belongs to $\mathcal{N}_\lambda$. Under these conditions, as was quoted in the introduction, we shall split the Nehari manifold $\mathcal{N}_\lambda$ into three disjoint subsets in the following way:
\begin{eqnarray}
\mathcal{N_\lambda}^+&=&\{u\in\mathcal{N}_\lambda:\phi''(1)=E''_\lambda(u)(u,u)>0 \}\label{n+}\\ 
\mathcal{N_\lambda}^-&=&\{u\in\mathcal{N}_\lambda:\phi''(1)=E''_\lambda(u)(u,u)<0 \}\label{n-}\\
\mathcal{N_\lambda}^0&=&\{u\in\mathcal{N}_\lambda:\phi''(1)=E''_\lambda(u)(u,u)=0\}\label{n0}.
\end{eqnarray}
More generally, we see that $t u \in \mathcal{N}_\lambda$ if and only if $\phi'(t) = E'_\lambda(t u) u = 0$. This can be checked using the definition just above together with the claim rule. Hence critical points for the fibering map $\phi$ give us functions in the Nehari manifold. At this stage, we shall define the following set
\begin{eqnarray}\label{conjE}
\mathcal{E}&:=&\{u\in X\setminus \{0\}:E_\lambda(u)=0 \} \nonumber \\
&=&\left\{ u\in  X\setminus \{0\}:\dfrac{1}{2}||u||^2-\dfrac{1}{2p}\Int\left(I_\alpha*|u|^p\right)|u|^pdx=\dfrac{\lambda}{q}{\Int|u|^qdx}\right\}.
\end{eqnarray}

It is important to stress that equations \eqref{nehari} and \eqref{conjE} allow us to define the nonlinear generalized Rayleigh
quotients which have been explored in the last years, see \cite{yavdat1}. More specifically, we consider the functionals $R_n,R_e: X\setminus\{0\} \to \mathbb{R}$ 
associated with the parameter $\lambda>0$ in the following form 
\begin{equation}\label{Rn}
R_n(u)=\dfrac{||u||^2-\Int\left(I_\alpha*|u|^p\right)|u|^pdx}{\Int|u|^qdx}, u \in X \setminus \{0\}
\end{equation}
and
\begin{equation}\label{Re}
{R}_e(u)=\dfrac{\frac{1}{2}||u||^2-\frac{1}{2p}\Int\left(I_\alpha*|u|^p\right)|u|^pdx}{\frac{1}{q}\Int|u|^qdx}, u \in X \setminus \{0\}.
\end{equation}
It is easy to ensure that $R_e,R_n$ belongs to $C^1(X \setminus \{0\}, \mathbb{R})$. Moreover, the functionals $R_e,R_n$ are related with the energy functional $E_\lambda$ and its derivatives. Hence we consider the following informations:

\begin{rmk}\label{rmk1}
Let $u\in X\setminus\{0\}$ be fixed. It is not hard to verify that \eqref{nehari} implies
		\begin{itemize}\item[i)]$R_n(u)=\lambda$ if and only if $E_\lambda'(u)u=0$,
			\item[ii)]$R_n(u) > \lambda$ if and only if $E_\lambda'(u)u > 0$,
			\item[iii)]$R_n(u) < \lambda$ if and only if $E_\lambda'(u)u < 0$.
			\end{itemize}
	\end{rmk}
Similarly, we also consider the following remark:
\begin{rmk}\label{rmk11} Let $u\in X\setminus\{0\}$ be fixed. It is easy to verify that \eqref{conjE} implies
\begin{itemize}
\item[i)] ${R}_e(u)=\lambda$ if and only if $E_\lambda(u)=0$, 
\item[ii)]$R_e(u) > \lambda$ if and only if $E_\lambda (u) > 0$.
\item[iii)]$R_e(u) < \lambda$ if and only if $E_\lambda (u) < 0$.
\end{itemize}
\end{rmk}

Now we define the auxiliary functional $G: X \rightarrow \mathbb{R}$ given by $G(u) = \int_{\mathbb{R}^N} |u|^q dx, u \in X$. In this way, we also mention that 
\begin{equation}\label{aee}
\dfrac{d}{d t} R_e(tu) = \dfrac{E'_{\lambda}(tu) u}{G(tu)}, t > 0, 
\end{equation}
holds for each $u \in X \setminus \{0\}$ such that $R_e(tu) = \lambda$. In particular, we obtain the following result

\begin{prop}\label{importante} Suppose $(Q)$ and $(V_1)-(V_2)$. Assume also that $u \in X \setminus \{0\}$ satisfies $R_e(tu) = \lambda$ for some $t > 0$. Then we obtain the following assertions: 
\begin{itemize}
	\item[$i)$] $\dfrac{d}{d t} R_e(tu)  > 0$ if and only if $E'_{\lambda}(tu) tu > 0$,
	\item[$ii)$] $\dfrac{d}{d t} R_e(tu)  < 0$ if and only if $E'_{\lambda}(tu) tu < 0$,
	\item[$iii)$] $\dfrac{d}{d t} R_e(tu)  = 0$ if and only if $E'_{\lambda}(tu) tu = 0$.
\end{itemize}
\end{prop}
\begin{proof}
The proof follows immediately using \eqref{aee} the fact that that $R_e(tu) = \lambda$ for some $t > 0$ where $u \in X \setminus \{0\}$, see \cite{marcos,yavdat1}.  
\end{proof}
Analogously, we observe that 
\begin{equation}\label{save}
\dfrac{d}{d t} R_n(tu) = \frac{1}{t} \dfrac{E''_{\lambda}(tu) (tu, tu)}{G'(tu)(tu)}, t > 0, 
\end{equation}
holds for each $u \in X \setminus \{0\}$ such that $R_n(tu) = \lambda$. In particular, we obtain the following result 
\begin{prop} \label{der-Rn}
	Suppose $(Q)$ and $(V_1)-(V_2)$. Assume also that $u \in X \setminus \{0\}$ satisfies $R_n(tu) = \lambda$ for some $t > 0$. Then we obtain the following assertions: 
	\begin{itemize}
		\item[$i)$] $\dfrac{d}{d t} R_n(tu)  > 0$ if and only if $E''_{\lambda}(tu) (tu, tu) > 0$,
		\item[$ii)$] $\dfrac{d}{d t} R_n(tu)  < 0$ if and only if $E''_{\lambda}(tu) (tu, tu) < 0$,
		\item[$iii)$] $\dfrac{d}{d t} R_n(tu) = 0$ if and only if $E''_{\lambda}(tu) (tu, tu) = 0$.
	\end{itemize}
\end{prop}
\begin{proof}
The proof follows from \eqref{save} together with the identity $R_n(tu) = \lambda$ for some $t > 0$ with $u \in X$, see \cite{marcos,yavdat1}.
\end{proof}

At this stage, we consider the fibering function for each $t>0$ given by
\begin{equation*}
Q_n(t)=R_n(tu)=\dfrac{t^2||u||^2-t^{2p}\Int\left(I_\alpha*|u|^p\right)|u|^pdx}{t^q\Int|u|^qdx}=\dfrac{1}{||u||^q_q}\left[t^{2-q}||u||^2-t^{2p-q}\Int\left(I_\alpha*|u|^p\right)|u|^pdx\right].
\end{equation*}
As a consequence, we obtain the following identity 
\begin{equation}\label{derivada}
Q_n'(t)= \dfrac{1}{||u||^q_q}\left[(2-q)t^{1-q}||u||^2-(2p-q)t^{2p-q-1}\Int\left(I_\alpha*|u|^p\right)|u|^pdx\right].
\end{equation}
It is not difficulty to determine that the unique critical point of $Q_n$ is given by
  \begin{equation}\label{tn}
	  t_n(u)=\left[\dfrac{(2-q)}{(2p-q)}\dfrac{||u||^2}{\Int\left(I_\alpha*|u|^p\right)|u|^pdx}\right]^{\frac{1}{2p-2}}.
 \end{equation}
 Under these conditions, we obtain also that 
  	 \begin{equation}\label{Qtc}
  	 Q_n(t_n(u))=C_{p,q}\dfrac{||u||^{\frac{2p-q}{p-1}}}{||u||^q_q\left[\Int (I_\alpha*|u|^p)|u|^pdx\right]^{\frac{2-q}{2p-2}}}
  	 \end{equation}
  where $$C_{p,q} = \left(\frac{2-q}{2p-q}\right)^{\frac{2-q}{2p-2}}\left(\frac{2p-2}{2p-q}\right).$$
  Similarly, using the same ideas discussed just above, we show that
  	$Q'_n(t)>0$ for each $t\in(0,t_n(u))$ and $Q_n'(t)<0$ for each $ t>t_n(u).$
  Furthermore, we observe that $Q_n(0)=0$ and 
  \begin{equation}
  \lim_{t\to 0} \frac{Q_n(t)}{t^{2-q}} > 0,  \lim_{t\to+\infty}Q_n(t)=-\infty.
  \end{equation}
  	 Thus, the number $t_n(u)>0$ is the unique critical point for $Q_n$ and $Q_n(t_n(u))=\displaystyle\max_{t>0}Q_n(t)$. It is important to emphasize that $Q_e$ has a similar behavior. More precisely, we consider the function
\begin{equation*}
{Q}_e(t)={R}_e(tu)=\dfrac{\dfrac{t^2}{2}||u||^2-\dfrac{t^{2p}}{2p}\Int\left(I_\alpha*|u|^p\right)|u|^pdx}{\dfrac{t^q}{q}\Int|u|^qdx}=\dfrac{q}{||u||^q_q}\left[\dfrac{t^{2-q}}{2}||u||^2-\dfrac{t^{2p-q}}{2p}\Int\left(I_\alpha*|u|^p\right)|u|^pdx\right].
\end{equation*}
As a consequence, the derivative is given by the following identity
\begin{equation}\label{derivada1}
{Q}_e'(t)= \dfrac{q}{||u||^q_q}\left[\dfrac{(2-q)}{2}t^{1-q}||u||^2-\dfrac{(2p-q)}{2p}t^{2p-q-1}\Int\left(I_\alpha*|u|^p\right)|u|^pdx\right]
\end{equation}
As was done before we show that ${Q}_e'(t)=0$ if and only if $t = t_e(u)$ where
\begin{equation}\label{T}
t_e(u)=\left[p\dfrac{(2-q)}{(2p-q)}\dfrac{||u||^2}{\Int\left(I_\alpha*|u|^p\right)|u|^pdx}\right]^{\frac{1}{2p-2}}.
\end{equation}
Once more we ensure that ${t}_e(u) > 0$ is the unique critical point of ${Q}_e$ and $ {Q_e}({t}_e(u))=\displaystyle\max_{t>0}{Q_e}(t)$. Here we mention that  
\begin{equation}
	{Q}_e({t}_e(u))=\tilde{C}_{p,q}\dfrac{||u||^{\frac{2p-q}{p-1}}}{||u||^q_q\left[\Int (I_\alpha*|u|^p)|u|^pdx\right]^{\frac{2-q}{2p-2}}}
\end{equation} 
where $$\tilde{C}_{p,q}=\left(p\dfrac{2-q}{2p-q}\right)^{\frac{2-q}{2p-2}}\left(q\dfrac{p-1}{2p-q}\right) > 0.$$

\begin{rmk}\label{Qe=Qn}
	It is worthwhile to mention that 
	 $R_n(tu)={R}_e(tu)$ if and only if $t={t}_e(u)$.
 In fact, the identity $R_n(tu)={R_e}(tu)$ is equivalent to the following expression
$$ t^{2-q}||u||^2\left(\dfrac{2-q}{2}\right)-t^{2p-q}\left(\dfrac{2p-q}{2p}\right)\Int(I_\alpha*|u|^p)|u|^pdx=0.$$
In view of the last identity we also obtain
$$
0=\dfrac{q}{||u||_q^q}\left(t^{1-q}||u||^2\left(\dfrac{2-q}{2}\right)-t^{2p-q-1}\left(\dfrac{2p-q}{2p}\right)\Int(I_\alpha*|u|^p)|u|^pdx\right)=({R}_e(tu))', t > 0.
$$
As a consequence,  $R_n(tu)={R}_e(tu)$ if and only if $t={t}_e(u)$. Furthermore, we also mention that $R_n(tu) > R_e(tu)$ for each $t \in (0, t_e)$. In the same way, we observe that $R_n(tu)  < R_e(tu)$ for each $t \in (t_e, \infty)$, see Figure 4. 
\end{rmk}
Under these conditions we are able to ensure the following result

\begin{lem}\label{Lambda} 
	Suppose $(Q)$ and $(V_1)-(V_2)$. Let $$\Lambda_n(u):= {Q_n}(t_n(u))=C_{p,q}\dfrac{||u||^{\frac{2p-q}{p-1}}}{||u||^q_q\left[\Int (I_\alpha*|u|^p)|u|^pdx\right]^{\frac{2-q}{2p-2}}}.$$
Then we obtain the following statements:
	\begin{itemize}
		\item[i)] $\Lambda_n(u) $ is $0$-homogeneous, i.e., $\Lambda_n(tu) = \Lambda_n(u)$ for each $t > 0, u \in X\setminus\{0\}$;
		\item[ii)] There exists $u \in X\setminus\{0\}$ such that $\lambda_{n}=\Lambda_n(u)=\displaystyle\inf_{v\in X} \Lambda_n(v)$. Furthermore, we obtain that $\lambda_n > 0$.
		\item[iii)] The function $u$ given by the previous item is a weak solutions for the following elliptic problem
		\begin{equation}\label{eq-lambda-n}
		\left\{
		\begin{array}{rcl}
		-2\Delta u  +2V(x) u &=& 2p (I_\alpha* |u|^p)|u|^{p-2}u+ q\lambda_n |u|^{q-2}u \, \mbox{ in }\,    \mathbb{R}^N,  \\
		u\in X&&
		\end{array}
		\right.
		\end{equation}	
	\end{itemize}
\end{lem}

\begin{proof}
The proof for item $i)$ follows immediately using the following identities 
$$\Lambda_n(tu)
=\dfrac{t^{\frac{2p-q}{p-1}-q-2p\frac{2-q}{2p-2}} C_{p,q}||u||^{\frac{2p-q}{p-1}}}{||u||^q_q\left[\Int (I_\alpha*|u|^p)|u|^pdx\right]^{\frac{2-q}{2p-2}}}=\Lambda_n(u), t > 0, u \in X\setminus\{0\}.$$

Now we shall prove the item $ii)$. In order to do that we shall show that $\Lambda_n$ is bounded from below. Indeed, by using the fact that $p< (N+\alpha)/(N-2)$ and Lemma \ref{lemaHardy}, there exists a positive constant $C_1$ such that  
\begin{equation}\label{P}
\Int (I_\alpha*|u|^p)|u|^pdx\leq C_1||u||^{2p}. 
\end{equation}
Now, using the fact that $X\hookrightarrow L^q(\Rn)$, for each $u \in X$ with $||u||=1$ it follows from item $i)$ and \eqref{P} that  
\begin{equation*}\label{liminferior}
	\Lambda_n(u)\geq C > 0
\end{equation*}
holds for some $C > 0$.
This fact implies that $\lambda_n > 0$  as was mentioned before. Now consider a minimizer sequence $(u_k) \in X\setminus\{0\}$, i.e, $\Lambda_n(u_k) \rightarrow \lambda_n$ as $n \rightarrow \infty$. Without any loss of generality, using the fact that $\Lambda_n$ is zero homogeneous, we assume that $(u_k)$ is normalized in $L^q(\mathbb{R}^N)$, i.e., we have $\|u_k\|_q =1$ for each $k \in \mathbb{N}$. Now we claim that that $(u_k)$ is bounded in $X$. In fact, by using the fact that $(u_k)$ is normalized in $L^q(\mathbb{R}^N)$ and \eqref{P}, we obtain
\begin{eqnarray}
\|u_k\|^{(2p - q)/(p -1)} &\leq& C \left( \Int (I_\alpha*|u_k|^p)|u_k|^pdx\right)^{(2 - q)/2(p -1)} 
\leq C \|u_k\|^{p(2- q)/(p -1)}.
\end{eqnarray}
The last estimates imply that $\|u_k\|^{q} \leq C$ for some $C > 0$. As a consequence, there exists $u \in X$ such that $u_k \rightharpoonup u$ in $X$. 
Furthermore, by using the compact embedding $X\hookrightarrow L^r(\mathbb{R}^N)$ for each $r\in [1,2^*)$, we obtain also that $\|u\|_q = 1$. The last assertion implies that $u \neq 0$. Now, using the fact that the norm $\|\cdot \|$ is weakly lower semicontinuous and the compact embedding $X\hookrightarrow L^r(\mathbb{R}^N)$ for each $r \in [1, 2^*)$, we observe that $\Lambda_n$ is also weakly lower semicontinuous. Hence, $\Lambda_n(u) \leq \liminf_{n \rightarrow \infty} \Lambda_n(u) = \lambda_n$. This ends the proof of item $ii)$.

Now we shall prove the item $iii)$. Since $\Lambda_n$ is attained we mention that	
	$$\lambda_n:=\inf_{v\in X\setminus\{0\}}\Lambda_n(v)=\Lambda_n(u):=R_n(t_n(u)u)$$
	holds for some $u \in X \setminus \{0\}$.
	Since $t_n(u)$ is the maximum point of $Q_n(t):=R_n(tu)$ we observe that 
	\begin{equation}\label{max-n}
	0=Q_n'(t_n(u))=(R_n)'(t_n(u)u)u.
	\end{equation}
	
	On the other hand, by using the fact that $u$ is a critical point of $\Lambda_n$, we infer that
	\begin{eqnarray}\label{Lambda_n}
	0&=&(\Lambda_n)'(u)w=(R_n(t_n(u)u))'w\nonumber\\
	&=&(R_n)'(t_n(u)u) \left[(t_n)'(u)w\right] u  +(R_n)'(t_n(u)u)t_n(u)w,~w\in X.
	\end{eqnarray}
	It follows from \eqref{max-n} and \eqref{Lambda_n} that $(R_n)'(t_n(u)u)w=0$ holds for all $w\in X$. Now, we define the auxiliary function $v := t_n(u)u \in X \setminus \{0\}$. Therefore, by using the fact that  
	$$\lambda_n = R_n(v) = \frac{\|v\|^2-\Int\left(I_\alpha*|v|^p\right)|v|^{p}dx}{\|v\|_q^q}$$
	we obtain the following identity  
	$$0=R'_n(v)w=\frac{1}{\|v\|_q^q}\left[2\langle v,w\rangle -2p\Int\left(I_\alpha*|v|^p\right)|v|^{p-2}v w dx-q \lambda_n \Int|v|^{q-2}v w dx\right] \,\, \mbox{for all} \,\, w \in X.$$
	The last assertion says that $v$ is a weak solution for the problem \eqref{eq-lambda-n}. This finished the proof.
\end{proof}

\begin{rmk}\label{impor}
	It is important to mention that using the function $\Lambda_e(u)=Q_e(t_e(u))$ instead of $\Lambda_n (u)=Q_n(t_n(u))$ we can prove that $\Lambda_e$ is also achieved by a function $u \in X$. Notice also that $\Lambda_e$ and $\Lambda_n$ are attained by the same function $u \in X$ which follows form the fact that $\Lambda_e(v) = C \Lambda_n(v)$ for each $v \in X$ where $C \in (0, 1)$. Furthermore, we observe that $0 < \lambda_e < \lambda_n <  \infty $. Clearly, the functional $\Lambda_e$ is also a zero homogeneous function. 
\end{rmk}

Now, by using Lemma \ref{Lambda} and Remark \ref{impor}, we can show that the fibering map $\phi(t)=E_\lambda(tu)$ has exactly two distinct critical points for each $\lambda \in (0, \lambda_n)$, see Figure 4. More specifically, we prove the following useful result:  
\begin{prop} \label{compar}
	Suppose $(Q)$ and $(V_1)-(V_2)$. Then for each $\lambda \in(0,\lambda_n)$ and $u\in X\setminus \{0\}$ 
	the fibering $\phi(t)=E_\lambda(tu)$ has exactly two distinct critical points $0<t_\lambda^{n,+}(u)<t_n(u)<t_\lambda^{n,-}(u)$. Moreover, we consider the following statements:
	\begin{itemize}
		\item[i)] The functional $t_\lambda^{n,+}(u)$ is a local minimum point for the fibering map $\phi$ which satisfies $t_\lambda^{n,+}(u)u\in\mathcal{N}_\lambda^+$. Furthermore, the functional $t_\lambda^{n,-}(u)$ is a local maximum for the fibering map $\phi$ which verifies $t_\lambda^{n,-}(u)u\in\mathcal{N}_\lambda^-$.
		\item[ii)] The functions $u \mapsto t_\lambda^{n,+}(u)$ and $u \mapsto t_\lambda^{n,-}(u)$ belong to $C^1(X \setminus \{0\}, \mathbb{R})$.
	\end{itemize}
	\end{prop} 
\begin{proof} 
	Let $0<\lambda<\lambda_n$ and $u\in X\setminus \{0\}$ be fixed. In view of \eqref{lambdan} we mention that $R_n(t_n(u)u)=Q(t_n(u))\geq \lambda_n >\lambda$. Here we refer the reader to Figure \ref{figQnQe} where was used the fact that $Q_n(t_n(u))=\displaystyle\max_{t>0}Q_n(t)$.  As a consequence, we obtain that the identity $Q_n(t)=R_n(tu)=\lambda$ admits exactly two roots. Namely, we consider its roots in the following form $0<t_\lambda^{n,+}(u)<t_n(u)<t_\lambda^{n,-}(u)$. Clearly, the roots $t_\lambda^{n,+}(u)$ and $t_\lambda^{n,-}(u)$ are critical points for the fibering map $\phi(t)=E_\lambda(tu)$, see Remark \ref{rmk1}. Under these conditions, we observe also that 
	\begin{equation}\label{Q'}
	Q'_n(t_\lambda^{n,+}(u))>0 \mbox{ and } Q'_n(t_\lambda^{n,-}(u))<0.
	\end{equation} 
	
	On the other hand, by using \eqref{save} and considering $G(u)=||u||_q^q$, we mention that
	\begin{equation*}
		0<Q'_n(t_\lambda^{n,+}(u)) = \frac{1}{t_\lambda^{n,+}(u)} \dfrac{E''_{\lambda}(t_\lambda^{n,+}(u)u) (t_\lambda^{n,+}(u)u, t_\lambda^{n,+}(u)u)}{G'(t_\lambda^{n,+}(u)u)(t_\lambda^{n,+}(u)u)}, u \in X \setminus \{0\}.
	\end{equation*}
	Hence, we obtain that $E''_{\lambda}(t_\lambda^{n,+}(u)u) (t_\lambda^{n,+}(u)u, t_\lambda^{n,+}(u)u)>0$ holds, see Proposition \ref{der-Rn}.  Hence, by using \eqref{n+}, we deduce that $t_\lambda^{n,+}(u)u\in \mathcal{N}_\lambda^+$.
	In the same way, we conclude that $t_\lambda^{n,-}(u)u\in \mathcal{N}_\lambda^-$ holds true. These statements finish the proof of item $(i)$.
	\begin{figure}[!ht]\label{figQnQe}
		\begin{minipage}[h]{0.49\linewidth}
			\center{\includegraphics[scale=0.7]{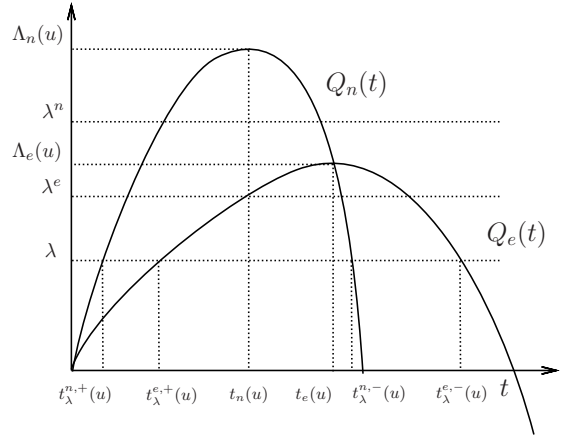}}
			\caption{The functions $Q_n(t)$, $Q_e(t)$}
		\end{minipage}
	\end{figure}

Now we shall prove the item $ii)$. Initially, we observe that $\lambda \in (0, \lambda_n)$ implies that $\lambda < Q_n(t_n(u)) = R_n(t_n(u) u)$ for all $u \in X \setminus \{0\}$. As was mentioned before we obtain exactly two roots for the equation $R_n(t u) = \lambda$. These roots satisfies $0 < t_\lambda^{n,+} < t_n(u) < t_n^{n,-}$ with $t_\lambda^{n,+} u \in \mathcal{N}_\lambda^+$ and $t_\lambda^{n,-} u \in \mathcal{N}_\lambda^-$. Hence, we obtain that $\mathcal{N}_\lambda=\mathcal{N}_\lambda^+\cup\mathcal{N}_\lambda^-$ is satisfied for each $\lambda \in (0, \lambda_n)$, see Proposition \ref{der-Rn}. Here we refer the interested reader to the important work \cite{yavdat1}. Under these conditions, by using the fact that $Q_n\in C^1(\mathbb{R^+},X)$ together with \eqref{Q'}, it follows from Implicit Function Theorem \cite{drabek} that functions $u \mapsto t_\lambda^{n,+}(u)$ and $u \mapsto t_\lambda^{n,-}(u)$ belong to $C^1(X \setminus \{0\}, \mathbb{R})$ for any $\lambda\in (0,\lambda_n)$. In fact, defining $L^{\pm} : (0, \infty) \times (X \setminus \{0\}) \rightarrow  \mathbb{R}$ given by $L^{\pm}(t, u) = E'_\lambda(tu) tu$, we obtain that $L^{\pm}(t, u) = 0$ if and only if $t u \in \mathcal{N}_\lambda$. Furthermore, we observe that $\frac{\partial}{\partial t} L^{\pm}(t, u) \neq 0$ for each $(t, u) \in (0, \infty) \times (X \setminus \{0\})$ such that $tu \in \mathcal{N}_\lambda^{\pm}$. This finishes the proof. 
\end{proof}

\begin{rmk}
Under assumptions of Proposition \ref{compar} it follows that $\mathcal{N}_{\lambda}^0$ is empty for each $\lambda \in (0, \lambda_n)$. This fact allows us to apply the Nehari method taking into account the uniqueness of the projections in $\mathcal{N}_\lambda^+$ and $\mathcal{N}_\lambda^-$. This is the main feature for the parameter $\lambda_n > 0$. In other words, the parameter $\lambda_n$ is the first positive number in such way that $\mathcal{N}_\lambda^0$ is not empty, see \cite{yavdat1}. 
\end{rmk}

Using the same ideas discussed in the proof of Proposition \ref{compar} we can ensure an analogous result for the functional $Q_e$ instead of $Q_n$. More specifically, we consider the following result
\begin{prop} \label{compar1}
	Suppose $(Q)$ and $(V_1)-(V_2)$. Then for each $\lambda \in(0,\lambda_e)$ and $u\in X\setminus \{0\}$ there are 
	two points $0<t_\lambda^{e,+}(u)<t_e(u)<t_\lambda^{e,-}(u)$ such that $t_\lambda^{e,-}(u)u, \ t_\lambda^{e,+}(u)u \in \mathcal{E}$ and $Q_e'(t_\lambda^{e,-}(u))<0<Q_e'(t_\lambda^{e,+}(u))$. Moreover,  we obtain that the functions $u \mapsto t_\lambda^{e,+}(u)$ and $u \mapsto t_\lambda^{e,-}(u)$ belong to $C^1(X \setminus \{0\}, \mathbb{R})$. Furthermore, we mention that  $0< t_\lambda^{n,+}(u) < t_\lambda^{e,+}(u)< t_n(u) < t_e(u) < t_\lambda^{n,-}(u) < t_\lambda^{e,-}(u) < \infty$ holds for each $\lambda \in (0, \lambda_e)$.
\end{prop} 

\section{Proof of Theorem \ref{theorm1}}
In this section we shall prove our first main result. In order to that we need to consider some auxiliary tools. Firstly, we shall consider the following useful result: 

\begin{lem}\label{coercive} Suppose $(Q)$ and $(V_1)-(V_2)$. Then the energy functional $E_\lambda$ is coercive in $\mathcal{N_\lambda}$ for each $\lambda>0$. In particular, the functional $E_\lambda$ is bounded from below in $\mathcal{N}_\lambda$.
\end{lem}
\begin{proof}
Notice that, for each $u \in\mathcal{N_\lambda},$ we obtain $$\Int (I_\alpha*|u|^p)|u|^pdx=||u||^2-\lambda\Int |u|^qdx.$$ Hence, using the embedding  $X\hookrightarrow L^r(\mathbb{R}^N)$ for each $r\in [1,2^*)$, we deduce that
\begin{eqnarray}
E_\lambda(u)&=& \frac{1}{2}\left(1-\frac{1}{p}\right)||u||^2-\lambda\left(\frac{1}{q}-\frac{1}{2p}\right)\Int |u|^qdx \nonumber \\
&\geq& C_1||u||^2-\lambda C_2||u||^q=||u||^2\left(C_1-\lambda C_2||u||^{q-2}\right)
	\end{eqnarray}
holds for some positive constants $C_1, C_2$. Now, by using the fact that $1<q<2$, we deduce also that $E_\lambda(u)\to +\infty$ as $||u||\to +\infty$ with $u \in \mathcal{N}_\lambda$. This completes the proof. 
\end{proof}

\begin{lem}\label{E-}
	Suppose $(Q)$ and $(V_1)-(V_2)$. Assume also that $\lambda\in(0,\lambda_n]$ holds.  Then, for each $u\in\mathcal{N_\lambda}^-\cup \mathcal{N}_\lambda^0$ there exists a constant $c = c(N,p,q)>0$ which does not depend on $\lambda$ in such way that $||u||\geq c$. In particular, the set $\mathcal{N_\lambda}^-\cup \mathcal{N}_\lambda^0$ is closed. 
\end{lem}
\begin{proof}
	Let $u\in\mathcal{N_\lambda}^-$ be a fixed function with $\lambda \in (0, \lambda_n)$. For this case we know that $t_\lambda^{n,-}(u) = 1$. It follows from \eqref{tn} and \eqref{P} that 
	\begin{equation}\label{eq11}
		1=t_\lambda^{n,-}(u)\geq t_n(u)=\left[\dfrac{(2-q)}{(2p-q)}\dfrac{||u||^2}{\Int\left(I_\alpha*|u|^p\right)|u|^pdx}\right]^{\frac{1}{2p-2}}\geq \left[C||u||^{2-2p}\right]^{\frac{1}{2p-2}}.
		\end{equation}
	These inequalities implies that $\|u\|\geq c_1$ holds true for some $c_1 = c_1(N,p,q) > 0$ with $\lambda \in (0, \lambda_n)$. Notice that $\mathcal{N}^0_\lambda = \emptyset$ for each $\lambda \in (0, \lambda_n)$. For the case $\lambda = \lambda_n$ we also observe that $\mathcal{N}_{\lambda_n}^0 \neq \emptyset$. Now we assume that $u \in \mathcal{N}_{\lambda_n}^-$ which occurs whenever $\lambda_n < \Lambda_n(u)$. The last statement implies that $Q'_n(1) < 0$, see Proposition \ref{der-Rn}. As a consequence we obtain that $t_n(u) < 1$. Hence using the same ideas discussed in \eqref{eq11} we infer that $\|u\| \geq c_2$ for each $u \in \mathcal{N}_{\lambda_n}^-$ holds for some $c_2 = c_2(N,p,q)$. Analogously, for each $u \in \mathcal{N}_{\lambda_n}^0$ we obtain that $R_n'(u) u = 0$, see Proposition \ref{der-Rn}. Recall that $u \in \mathcal{N}_{\lambda_n}^0$ whenever $\lambda_n = \Lambda_n(u)$.  Since $t_n(u)$ is unique maximum point for the function $Q_n$ it follows that $t_{n}(u) = 1$. As a consequence, using the same ideas employed in \eqref{eq11}, we obtain that $\|u\|\geq c_3$ holds true for some $c_3 = c_3(N,p,q) > 0$ where $u \in \mathcal{N}_{\lambda_n}^0$. Therefore, we obtain that $\|u\| \geq c$ for any $u \in \mathcal{N_\lambda}^-\cup \mathcal{N}_\lambda^0$ where $c = \min(c_1, c_2, c_3)> 0$. Using the strong convergence for sequences in $X$ together with the last estimate we obtain that the set $\mathcal{N_\lambda}^-\cup \mathcal{N}_\lambda^0$ is closed.  This ends the proof. 
\end{proof}

For the next result we shall prove that $\mathcal{N}_\lambda$ is a natural constraint for our main problem for each $\lambda \in (0, \lambda_n)$. More precisely, we show the following result
\begin{lem}\label{criticalpoint}
Suppose $(Q)$ and $(V_1)-(V_2)$. Let $u \in X$ be a local minimum (or local maximum) for $E_\lambda$ on $ \mathcal{N_\lambda}$, that is, assume that $u \in \mathcal{N}_\lambda^- \cup \mathcal{N}_\lambda^+$ is a minimizer for $E_{\lambda}$. Then $u$ is a free critical point of $E_\lambda$ on $X$, that is, we obtain that $E_{\lambda}'(u) \psi = 0$ for each $\psi \in X$ with $\lambda\in(0,\lambda_n]$.
\end{lem}
\begin{proof}
The proof follows using standard minimization arguments on the Nehari set $\mathcal{N}_\lambda$. Let $u \in X$ be a local minimum (or local maximum) for $E_\lambda$ on $ \mathcal{N_\lambda}$ where $\lambda\in(0,\lambda_n)$. Since $E''(u)(u,u) = \phi'_u(1) \neq 0$ it follows from Lagrange Multipliers Theorem that $E_\lambda'(u)\psi = 0$ for each $\psi \in X$. 
This ends the proof. 
\end{proof}

For the next result we shall prove that any minimizers sequences in $\mathcal{N}_\lambda^-$ converge strongly in $X$. As a consequence, any minimizer sequences allow us to find a critical point for the energy functional $E_\lambda$. Namely, we can prove the following result
\begin{lem}\label{converg}
Suppose $(Q)$ and $(V_1)-(V_2)$. Assume also that $\lambda \in (0, \lambda_n)$ holds. Let $(v_k)\subset\mathcal{N_\lambda}^- $ be a minimizer sequence. Then there exists $v_\lambda\in X\setminus\{0\}$ such that, up to a subsequence, $v_k\to v_\lambda$ in $X$ where $v_\lambda\in\mathcal{N_\lambda}^-$. Furthermore, we obtain that $\mathcal{E}_\lambda ^2=E_\lambda(v_\lambda)$.
\end{lem}
\begin{proof}
	Initially, we shall consider a minimizer sequence $(v_k)\subset\mathcal{N_\lambda}^-, $ i.e.,
	$E_\lambda(v_k)=\mathcal{E}_\lambda^2+o_k(1)$. 
	In fact, up to a subsequence, we assume $v_k \rightharpoonup v_\lambda$ in $X$. As a consequence using the compact embedding $X\hookrightarrow L^r(\mathbb{R}^N)$ for each $r\in [1,2^*)$, we infer that $v_k\to v_\lambda$ in $L^r(\mathbb{R}^N)$ and $v_k(x)\to v_\lambda(x) $ a.e. in $ \mathbb{R}^N$. Moreover, there exists $h_r \in L^r(\mathbb{R}^n)$ such that $|v_n|\leq h_r$ a.e in $\mathbb{R}^N$.  Under these conditions, we infer also that 
	 \begin{equation}\label{limites}
	\int (I_\alpha*|v_k|^p)|v_k|^pdx=\Int (I_\alpha*|v_\lambda|^p)|v_\lambda|^pdx+o_k(1) \mbox{ and } \Int |v_k|^qdx=\Int |v_\lambda|^qdx+o_k(1).
	\end{equation} 
At this stage we observe that $ v_\lambda \neq 0 $. Indeed, arguing by contradiction we assume that $v_\lambda \equiv 0$ and $v_k \rightharpoonup 0$ in $X$.  Now, we define the normalized sequence $w_k= v_k/||v_k||\in S^1$ where $S^1$ is the unit sphere of $X$. Thus,  we write $v_k=||v_k||w_k=t_\lambda^{n,-}(w_k)w_k$. Here was used the fact that $t_\lambda^{n,-}(w_k) = ||v_k||$ which can be proved using Proposition \ref{compar}. For simplicity we now write $t_k = t_\lambda^{n,-}(w_k)$. In view of Lemma \ref{E-} taking into account that $v_k\in\mathcal{N_\lambda}^-$ there exists a constant $c>0$ such that $t_k \geq c > 0$. Therefore, $t_k\to t_0>0$ as $k \to \infty.$ Hence, $ v_k \rightharpoonup 0 $  is equivalent to $w_k \rightharpoonup 0$ in $X$. Under these conditions, by using the fact that $v_k\in\mathcal{N_\lambda}^-$ and taking into account Remark \ref{rmk1}, we get
	\begin{equation*}
	\dfrac{\lambda}{t_0^{2-q}} + o_k(1) =\frac{\lambda}{t_k^{2-q}} =\dfrac{1-t_k^{2p-2}\Int(I_\alpha*|w_k|^p)|w_k|^p}{||w_k||_q^q}+o_k(1). 
	\end{equation*}
	As a consequence, for each $\epsilon>0, $ there exists $k_0\in\mathbb{N}$ such that
	\begin{equation}\label{ed44}
	1-t_k^{2p-2}\Int(I_\alpha*|w_k|^p)|w_k|^p<\left(\dfrac{\lambda}{t_0^{2-q}}+\epsilon\right)
	||w_k||_q^q
	\end{equation}
	 holds for each $k>k_0$. On the other hand, by using the fact that $w_k \rightharpoonup 0$, the compact embedding $X\hookrightarrow L^r(\mathbb{R}^N)$ for each $r\in [1,2^*)$ together with Lemma \ref{lemaHardy} imply that
	\begin{equation*}
	\int (I_\alpha*|w_k|^p)|w_k|^pdx = o_k(1), \Int |w_k|^qdx = o_k(1).
	\end{equation*} 
This is a contradiction with \eqref{ed44} proving that $v_\lambda\neq 0.$

Recall also that, by using Proposition \ref{compar}, we obtain that the fibering map $\phi(t)=E_\lambda(tu), t \geq 0$ admits an unique critical point $t^{n,-}_\lambda(v_\lambda)>0$ in such way that $t^{n,-}_\lambda(v_\lambda)v_\lambda\in \mathcal{N_\lambda}^-.$ Now, arguing by contradiction, we assume that $v_k$ does not converge to $v_\lambda$ in $X$. As a consequence, we infer that  $||v_k|| < \liminf ||v_k||$. Since $(v_k) \in \mathcal{N_\lambda}^-$ we also mention $E_\lambda(v_k)\geq E_\lambda(sv_k)$ holds for any $s\geq t_\lambda^{n,+}(v_k)$. Now we claim that 
$t_\lambda^{n,-}(v_\lambda) > t_\lambda^{n,+}(v_k)$. The proof for the claim follows using the fact that $v \mapsto E'_\lambda(v)v$ is weakly lower semicontinuous. Indeed, we obtain that
\begin{equation}
0 = E'_\lambda(t_\lambda^{n,-}(v_\lambda)v_\lambda) v_\lambda < \liminf E'_\lambda(t_\lambda^{n,-}(v_\lambda)v_k)v_k.
\end{equation} 
As a consequence, $E'_\lambda(t_\lambda^{n,-}(v_\lambda)v_k)v_k > 0$ for any $k$ large enough. The last statement shows that $t_\lambda^{n,-}(v_\lambda) \in (t_\lambda^{n,+}(v_k),t_\lambda^{n,-}(v_k))$, see Figures 1, 2 and 3.
Therefore, using the last inequality together with the fact that $v_k$ does not converge to $v_\lambda$  in $X$ , we deduce that 
$$
\begin{array}{rcl}
E_\lambda({t^{n,-}_\lambda(v_\lambda)v_\lambda})< \liminf E_\lambda({t^{n,-}_\lambda(v_k)v_k})\leq \lim E_\lambda(v_k)=\mathcal{E}^2_\lambda.
\end{array} 
$$ 
This is a contradiction due the fact that $t^{n,-}_\lambda(v_\lambda)v_\lambda\in \mathcal{N_\lambda}^-$. To sum up, we have been showed that $v_k\to v_\lambda$ in $X$. Using the strong convergence in $X$ it follows also that $ \mathcal{E}_\lambda ^2=\lim E_\lambda(v_k)=E_\lambda(v_\lambda)$. This ends the proof.
\end{proof}

\begin{lem}\label{sing-u}
	Suppose $(Q)$ and $(V_1)-(V_2)$. Assume also that $\lambda\in (0,\lambda_n)$. Then $\mathcal{E}_\lambda^1 =E_\lambda(u_{_{\lambda}}) < 0$.
\end{lem}
\begin{proof}
	Let $\lambda\in (0,\lambda_n)$ be fixed. It is easy to see that $R_e(tu)<R_n(tu),~t\in(0,t_{e}(u))$
	holds for all $u\in X\setminus\{0\}$. Recall also that $t_\lambda^{n,+}(u) < t_e(u)$. Hence we also see that  $R_e(t_\lambda^{n,+}(u)u)<R_n(t_\lambda^{n,+}(u)u)=\lambda$. The last assertion implies that $E_\lambda(t_\lambda^{n,+}(u)u)<0$, see Remark \ref{rmk11}. Moreover, by using the fact that $t_\lambda^{n,+}(u)u\in \mathcal{N}_\lambda^+$, we infer that  
	$$\mathcal{E}_\lambda^1=E_\lambda(u_{_{\lambda}})=\inf_{u\in \mathcal{N}_\lambda^+}E_\lambda(u)\leq E_\lambda(t_\lambda^{n,+}(u)u)<0. $$
	This ends the proof. 
\end{proof}

\begin{lem}\label{converg1}
	Suppose $(Q)$ and $(V_1)-(V_2)$ and  $\lambda \in (0, \lambda_n)$. Let $(u_k)\subset\mathcal{N_\lambda}^+ $ be a minimizer sequence for $E_\lambda$ in $\mathcal{N}_\lambda^+$. Then, there exists $u_\lambda\in X\setminus \{0\}$ such that, up to a subsequence, $u_k\to u_\lambda$ in $X$ where $u_\lambda\in\mathcal{N_\lambda}^+$. Furthermore, we obtain that $\mathcal{E}_\lambda^1= E_\lambda(u_\lambda)$.
\end{lem}
\begin{proof}
	Initially, up to a subsequence, arguing as done in the proof of Lemma \ref{converg} we obtain we obtain that $u_k\to u_\lambda$ in $L^r(\mathbb{R}^N)$, $u_k(x)\to u_\lambda(x)$  and $|u_k|\leq h_r$ a.e. in $ \mathbb{R}^N$ where $h_r \in L^r(\mathbb{R}^N), r \in [1, 2^*)$.  Furthermore, by using the fact that $u_k\in\mathcal{N_\lambda}^+$, we also obtain
$$	\begin{array}{rcl}
	\lambda\Int|u_k|^qdx=\left(\dfrac{2q}{2-q}\right)\left(\dfrac{2(p-1)}{4p}\right)\Int (I_\alpha*|u_k|^p)|u_k|^pdx-\dfrac{2q}{2-q}E_\lambda(u_k)
\end{array}.$$ 
Taking into account that $(u_k)\subset\mathcal{N_\lambda}^+ $ is a minimizer for $E_\lambda$, together with hypothesis $(Q)$ and \eqref{limites}, we infer that 
$$\lambda \Int |u_\lambda|^qdx\geq- \dfrac{2q}{2-q}\mathcal{E}_\lambda^1>0.$$ In particular, we obtain that $u_\lambda\neq 0$. Notice also that $\mathcal{E}_\lambda^1$ is negative, see  Lemma \ref{sing-u}. 

From now on the proof of the strong convergence $u_k\to u_\lambda$ in $X$ follows arguing by contradiction. Assume that $u_k$ does not converge to $u_\lambda$ in $X$. In particular, we observe that 
$||u_\lambda|| < \liminf ||u_k||$. According to Proposition \ref{compar} there exists an unique $t_\lambda^{n,+}(u_\lambda)>0$ such that $t_\lambda^{n,+}(u_\lambda)u_\lambda\in \mathcal{N_\lambda}^+$. Furthermore, we know that  $\phi'(t_\lambda^{n,+}(u_\lambda))=E'_\lambda(t_\lambda^{n,+}(u_\lambda)u_\lambda)u_\lambda=0$ and $\phi(t_\lambda^{n,+}(u_\lambda))=E_\lambda(t_\lambda^{n,+}(u_\lambda)u_\lambda)<0.$ As a consequence, by using again the compact embedding $X\hookrightarrow L^r(\mathbb{R}^N)$ for each $r\in [1,2^*)$ and the fact that $(u_k)\subset \mathcal{N_\lambda}^+$, we obtain
  \begin{equation}
\frac{d}{dt}E_\lambda (tu_\lambda) = E'_\lambda(t u_\lambda) u_\lambda < \liminf E'_\lambda(t u_k) u_k = \frac{d}{dt}E_\lambda(tu_k) \leq 0
\end{equation}
holds for any $t\in (0,1]$. Here was used the fact that $E'_\lambda(u_k) u_k=0$ and $E''_\lambda(u_k) (u_k, u_k) > 0$ holds for each $k \in \mathbb{N}$. The last assertion implies also that $t_\lambda^{n,+}(u_\lambda) \geq 1$.  Under these conditions, using that $t_\lambda^{n,+}(u_\lambda)u_\lambda\in\mathcal{N_\lambda}^+$,  we deduce 
 $$E_\lambda(t_\lambda^{n,+}(u_\lambda)u_\lambda)\leq E_\lambda(u_\lambda)<\liminf E_\lambda(u_k)=\mathcal{E}^1_\lambda.$$ This is a contradiction proving that $u_k\to u_\lambda$ in $X$. This finishes the proof. 
\end{proof}

\begin{prop}\label{exx}
	Suppose $(Q)$ and $(V_1)-(V_2)$. Then the energy functional $E_{\lambda}$ admits at least two critical points $u_{\lambda}$ and $v_{\lambda}$ for each $\lambda \in (0, \lambda_n)$. Furthermore, $u_\lambda$ and $v_\lambda$ are strictly positive in $\mathbb{R}^N$.
\end{prop}
\begin{proof}
In view of Proposition \ref{coercive} we know that $E_\lambda$ is coercive and bounded from below in $\mathcal{N_\lambda}^-$. Let $(v_k)$ be a minimizer sequence for $E_\lambda$ in $\mathcal{N_\lambda}^-$. It is easy to see that $(v_k)$ is bounded in $X$. Up to a subsequence there exists $v_\lambda\in X$ such that
$v_k\rightharpoonup v_\lambda \,\, \mbox{in } \,\, X$.
It follows from the Lemma \ref{converg} that $v_k\rightarrow v_\lambda$ in  $X$.
Moreover, the last assertion says also that 
$\mathcal{E}_\lambda ^2=\lim E_\lambda(v_k)=E_\lambda(v_\lambda)$ holds for each $\lambda \in (0, \lambda_n)$. Hence we obtain that $v_\lambda \neq 0$. Therefore, by using Lemma \ref{criticalpoint}, we obtain that $v_\lambda$ is a weak solution to Problem \eqref{eq1}. Since the functional $E_\lambda$ is even we know that $E_\lambda(v_\lambda) = E_\lambda(|v_\lambda|)$. Furthermore, we also obtain that
$
E'_\lambda(|v_\lambda|) |v_\lambda| = E'_\lambda(v_\lambda) v_\lambda = 0.
$
As a consequence, we mention that $ |v_\lambda| \in \mathcal{N_\lambda}$. Recall that $v \mapsto E''_\lambda(v) (v,v)$ is an even functional. Hence $E''_\lambda(|v_\lambda|) (|v_\lambda|,|v_\lambda|) = E''_\lambda(v_\lambda) (v_\lambda,v_\lambda) < 0$ proving that 
$|v_\lambda| \in \mathcal{N_\lambda}^{-}$. Then $|v_\lambda|$ is a local minimizer in $\mathcal{N}_\lambda^-$ showing that $|v_\lambda|$ is now a critical point for the energy functional $E_\lambda$. Hence, without any loss of generality, we assume that $v_\lambda\geq 0$ in $\mathbb{R}^N$. Then the functional $E_\lambda$ admits at least one critical point $v_\lambda \in X$ for each $\lambda \in (0, \lambda_n)$ which satisfies $v_\lambda \geq 0$ in $\mathbb{R}^N$. Now we infer that $v_\lambda \in C^{1, \beta}_{loc}(\mathbb{R}^N)$ for some $\beta \in (0, 1)$, see Theorem \ref{regular} in Appendix. Now we shall prove that $v_\lambda$ is strictly positive. Arguing by contraction we assume that there exists $x_0 \in \mathbb{R}^N$ such that $v_\lambda(x_0) = 0$. Notice also that $v_\lambda \in C^{1,\beta} (B_r(x_0))$ for some $\beta \in (0,1)$ and for each $r > 0$. Therefore $v_\lambda$ satisfies the following inequalities
\begin{equation}
\left\{
\begin{array}{rcl}
-\Delta v_\lambda + V(x) v_\lambda &\geq& 0 \, \mbox{ in }\,    B_r(x_0),  \\
v_\lambda &\geq& 0  \, \mbox{on} \, \partial B_r(x_0).
\end{array}
\right.
\end{equation}
Hence the strong maximum principle for elliptic operators of second order on bounded domains implies that $v_\lambda > 0$ in $B_r(x_0)$ or $v_\lambda \equiv 0$ in $B_r(x_0)$. Assuming that $v \equiv 0$ in $B_r(x_0)$ and using the fact that $r > 0$ is arbitrary we
ensure that $v_\lambda \equiv 0$ in $\mathbb{R}^N$. This is a contradiction due the fact that $\|v_\lambda\| \geq c > 0$ for each $\lambda \in (0, \lambda_n)$, see
Lemma \ref{E-}. Hence we obtain that $v_\lambda > 0$ in $\mathbb{R}^N$.

Now we shall prove that the functional $E_\lambda$ has another critical point.
Arguing as was done just above we know that $E_\lambda$ is coercive and bounded from below in $\mathcal{N_\lambda}^+$. Let $(u_k)$ be a minimizer sequence for $E_\lambda$ in $\mathcal{N_\lambda}^+$. It is easy to see that $(u_k)$ is also bounded in $X$. Up to a subsequence there exists $u_\lambda\in X$ such that
$u_k\rightharpoonup u_\lambda \,\, \mbox{in } \,\, X$.
It follows from the Lemma \ref{converg1} that $u_k\rightarrow u_\lambda$ in $X$.
Under this condition we mention that
$$\mathcal{E}_\lambda ^1=\lim E_\lambda(u_k)=E_\lambda(u_\lambda).$$ 
Moreover, we obtain that $u_\lambda$ is a critical point for $E_\lambda$ for each $\lambda \in (0, \lambda_n)$. Recall also that $E_\lambda(u_\lambda) < 0$ for each $\lambda \in (0, \lambda_n)$, see Lemma \ref{sing-u}. The last assertion implies that $u_\lambda \neq 0$ and arguing as was done before we assume also that $u_\lambda > 0$ in $\mathbb{R}^N$. Since $\mathcal{N}_\lambda^- \cap \mathcal{N}_\lambda^+ =\emptyset$
we obtain that Problem \eqref{eq1} admits at least two positive solutions for each $\lambda \in (0, \lambda_n)$. This ends the proof.
\end{proof}

In order to prove our first main result we need to consider the case $\lambda=\lambda_e$. Under our assumptions we ensure existence of a critical point for the functional $E_\lambda$ with zero energy. More precisely, we consider the following result:

\begin{prop}\label{lambda-e}
Suppose $(Q)$ and $(V_1)-(V_2)$. Assume also that $\lambda=\lambda_e$ holds. Then the energy functional $E_{\lambda_e}$ admits a critical point  $w_{\lambda_{e}} \in X \setminus \{0\}$ such that $w_{\lambda_{e}}$ is a minimizer for the functional $\Lambda_e$.
\end{prop}
\begin{proof}
	Firstly, we recall that $\Lambda_e$ is attained, see Remark \ref{impor}. As a consequence, we also mention that	
		$$\lambda_e:=\inf_{v\in X\setminus\{0\}}\Lambda_e(v)=\Lambda_e(u):=R_e(t_e(u)u)$$
		holds for some $u \in X \setminus \{0\}$.
		Since $t_e(u)> 0$ is the maximum point of $Q_e(t):=R_e(tu)$ we observe that 
		\begin{equation}\label{max}
		0=Q_e'(t_e(u))=(R_e)'(t_e(u)u)u.
		\end{equation}
		
		On the other hand, by using the fact that $u$ is a critical point for the functional $\Lambda_e$, we infer that
		\begin{eqnarray}\label{Lambda_e}
		0&=&(\Lambda_e)'(u)w= [R_e(t_e(u)u)]'w\nonumber\\
		&=&(R_e)'(t_e(u)u)[t'_e(u) w] u + (R_e)'(t_e(u)u)t_e(u)w, \nonumber \\
		&=& [t'_e(u) w] [(R_e)'(t_e(u)u) u]  + (R_e)'(t_e(u)u)t_e(u)w,~w\in X.
		\end{eqnarray}
		It follows from \eqref{max} and \eqref{Lambda_e} that $(R_e)'(t_e(u)u)w=0$ holds for all $w\in X$. Now, we define the new function $w_{\lambda_e} := t_e(u)u$. Therefore, we obtain the following identities 
		$$0=R'_e(w_{\lambda_e})w= \|w_{\lambda_e}\|_q^{-q} E_{\lambda_e}'(w_{\lambda_e})w, \,\, \mbox{for all} \,\, w \in X.$$
		The last assertion says that $E_{\lambda_e}'(w_{\lambda_e})w = 0$ holds true for each $w \in X$. This finished the proof.
\end{proof}

At this stage, we shall consider the proof of Theorem \ref{theorm1}. This result depends on the signal for $E_\lambda(v_\lambda)$ according to the size of $\lambda > 0$. More precisely, we shall prove that $E_\lambda(v_\lambda) > 0$ for each $\lambda \in (0, \lambda_e)$. In this case, we shall prove also that  $E_\lambda(v_{\lambda_e}) = 0$ and $E_\lambda(v_\lambda) < 0$ for each $\lambda \in (\lambda_e, \lambda_n)$. These assertions are proved in the following way:

\nd{\bf The proof of Theorem \ref{theorm1} (i).}  Initially, for each $\lambda \in (0, \lambda_e)$ we obtain that Problem \eqref{eq1} admits at least two positive solutions $u_\lambda, v_\lambda \in X$, see Proposition \ref{exx}. Here we shall prove that $v_\lambda$ has positive energy for each $\lambda \in (0,\lambda_e)$, see Figure \ref{QnQemen-lam-e}. Now we claim that $t_\lambda^{n,-}(v_\lambda)=1>t_e(v_\lambda)$ holds for each $\lambda \in (0, \lambda_e)$. This can be done using the fact that there exist unique projections in the Nehari manifolds $\mathcal{N}_\lambda^-$ and $\mathcal{N}_\lambda^+$, respectively. 
\begin{figure}[!ht]
	\begin{minipage}[h]{0.49\linewidth}
		\center{\includegraphics[scale=0.6]{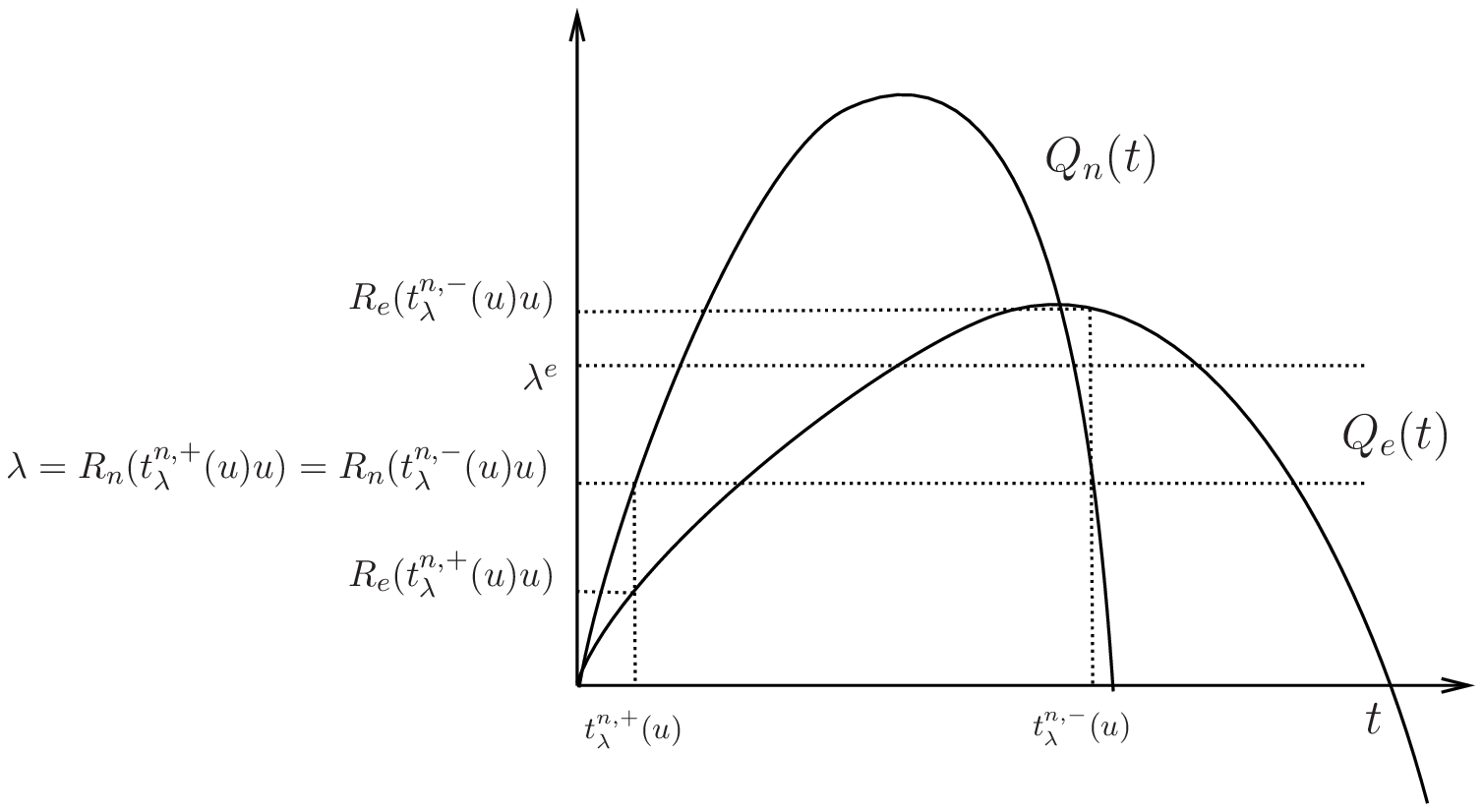}}
		\caption{$\lambda\in(0,\lambda_e)$}
		\label{QnQemen-lam-e}
	\end{minipage}
	\hfill
	\begin{minipage}[h]{0.49\linewidth}
		\center{\includegraphics[scale=0.6]{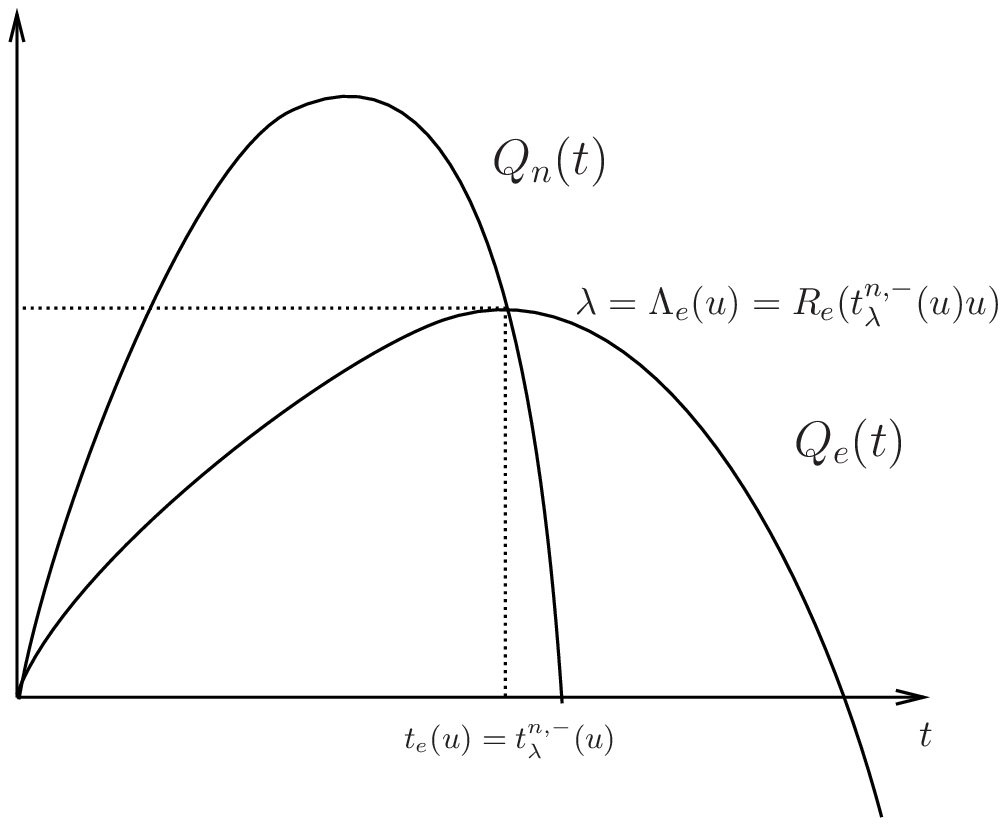}}
		\caption{ $\lambda=\lambda_e$}
		\label{QnQe-igu-lam-e}
	\end{minipage}
\end{figure}

On the other hand, we observe that $R_n(tv_\lambda)<R_e(tv_\lambda)$ holds for each $t>t_e(u)$. In particular, for $t=t_\lambda^{n,-}(v_\lambda)=1$, we get 
$\lambda=R_n(v_\lambda)=R_n(t_\lambda^{n,-}(v_\lambda)v_\lambda)<R_e(t_\lambda^{n,-}(v_\lambda)v_\lambda)=R_e(v_\lambda).$
The last assertion implies that $\mathcal{E}_\lambda^2=E_\lambda(v_\lambda)>0$, see Remark \ref{rmk11}.

\nd{\bf The proof of Theorem \ref{theorm1} (ii)}. Firstly, we put $\lambda=\lambda_e$. As was done in the proof of the previous item we obtain that Problem \eqref{eq1} admits at least two weak solutions $u_\lambda, v_\lambda \in X$, see Proposition \ref{exx}.   It follows from Proposition \ref{lambda-e} that $\lambda_e=\Lambda_e(w_{\lambda_e})$ for some $w_{\lambda_e} \in X$ where $w_{\lambda_e}$ is a critical point for the functional $E_\lambda$. In fact, the function $w_{\lambda_e}$ belongs to $\mathcal{N}_{\lambda}^-$, see Figure 6.  Now, we claim that $E_{\lambda_e}(w_{\lambda_e})=0$. Indeed, we observe that  
$t_{\lambda_e}^{n,-}(w_{\lambda_e})=t_e(w_{\lambda_e})=t_{\lambda_e}^{e,+}(w_{\lambda_e})=t_{\lambda_e}^{e,-}(w_{\lambda_e})=1$, see Figure \ref{QnQe-igu-lam-e}. Thus, we obtain the following identities 
$$\lambda_e=\Lambda_e(w_{\lambda_e})=R_e(w_{\lambda_e})=R_n(w_{\lambda_e})=R_e(t_{\lambda_e}^{n,-}(w_{\lambda_e})w_{\lambda_e}).$$
Therefore, the last assertion says that  $E_{\lambda_e}(w_{\lambda_e})=0$, see Remark \ref{rmk11}. In other words, we can find a critical point $w_{\lambda_e} \in X$ for the energy functional $E_{\lambda}$ with zero energy. In view of Proposition \ref{exx} we know also that 
\begin{equation}
E_{\lambda}(v_{\lambda}) = \inf_{w \in \mathcal{N}_\lambda^-} E_\lambda(w) \leq E_\lambda(w_{\lambda_e})= 0.
\end{equation}
Therefore $E_\lambda(v_\lambda) \leq 0$ holds true. The last estimate says that $R_e(v_\lambda) \leq \lambda_e$, see Remark \ref{rmk11}. Notice also that 
$
\lambda_e \leq \Lambda_e(v_\lambda) = R_e(t_e(v_\lambda )v_\lambda)$ holds true for some $t_e(v_\lambda) > 0$. Furthermore, by using the fact that $v_\lambda \in \mathcal{N}_\lambda^-$, one has $t_n(v_\lambda) <  t_{e}(v_\lambda)  \leq t_\lambda^{n,-}(v_\lambda) = 1$. As a consequence, we obtain that $\lambda = R_n(v_\lambda) \leq R_e(v_\lambda)$, see Figures 5 and 6. Here was used the fact that $R_n(t v) \leq R_e(t v)$ for each $t \geq t_e(v)$ with $v \in X \setminus \{0\}$. According to Remark \ref{rmk11} we obtain that $E_\lambda(v_\lambda) \geq 0$. The last estimate together with Remark \ref{rmk11} once more imply also that $$E_\lambda(v_\lambda) = 0, \quad \Lambda(v_\lambda) = \lambda_e$$ 
Hence we know that $E_{\lambda}(v_\lambda) = E_{\lambda}(w_{\lambda_e}) = 0$. As a consequence, the minimizer in $\mathcal{N}_\lambda^-$ has zero energy.   Notice also that $u_\lambda$ has negative energy, see Proposition \ref{sing-u}. This ends the proof.

\nd{\bf The proof of Theorem \ref{theorm1} iii).} As was mentioned before for each $\lambda \in (\lambda_e, \lambda_n)$ we obtain that Problem \eqref{eq1} admits at least two positive solutions $u_\lambda, v_\lambda \in X$, see Proposition \ref{exx}. Let $u\in X\setminus\{0\}$ be fixed such that 
$$\lambda_e\leq \Lambda_e(u)=R_e(t_e(u)u)<\lambda.$$
Using the last estimates we deduce that $t_\lambda^{n,-}(u)\in (0,t_e(u))$, see Figure \ref{QnQe-maior-lambda-e}. However, for each $t\in(0,t_e(u))$, we infer also that $R_e(tu)<R_n(tu)$. In particular, assuming that $t=t_\lambda^{n,-}(u)$ we get
$$R_e(t_\lambda^{n,-}(u)u)<R_n(t_\lambda^{n,-}(u)u)=\lambda.$$
The last assertion implies that $E_\lambda(t_\lambda^{n,-}(u)u)<0$, see Remark \ref{rmk11}. Furthermore, by using the fact that $t_\lambda^{n,-}(u)u\in \mathcal{N}_\lambda^-$, one has 
$\mathcal{E}_\lambda^2\leq E_\lambda(t_\lambda^{n,-}(u)u)<0.$
This ends the proof. 
\begin{figure}[!ht]
	\begin{minipage}[h]{0.49\linewidth}
		\center{\includegraphics[scale=0.6]{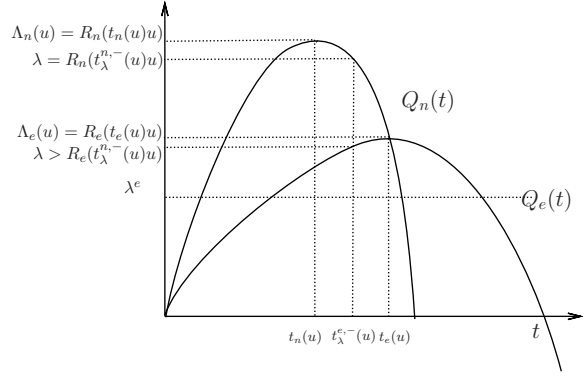}}
		\caption{$\lambda\in(\lambda_e,\lambda_n]$}
		\label{QnQe-maior-lambda-e}
	\end{minipage}
\end{figure}

\section{Proof of Theorems \ref{theorem2},  \ref{theorem3} and \ref{theorem4}}	
In this section we shall prove Theorems \ref{theorem2},  \ref{theorem3} and \ref{theorem4}. Firstly we consider the proof of Theorem \ref{theorem2}. Assume that $(\lambda_j)\subset (0,\lambda_n)$ such that  $\lambda_j\to\widetilde\lambda \in (0, \lambda_n)$ as $j \rightarrow \infty$. Recall that 
$\mathcal{E}^1_{\lambda}:=E_\lambda(u_\lambda)$ and $\mathcal{E}^2_{\lambda}:=E_\lambda(v_\lambda)$. 
Here we shall consider the proof for the function $\mathcal{E}^1_{\lambda}$. A similar proof can be done for the function $\mathcal{E}^2_{\lambda}$.

\begin{prop}\label{F}
Suppose $(Q)$ and $(V_1)-(V_2)$. Let $u_{\widetilde\lambda} \in \mathcal{N}_{\widetilde \lambda}^+$ be the weak solution for the Problem \eqref{eq1} for some $\widetilde \lambda \in (0, \lambda_n)$. Then $t_{\lambda_j}^{n,+}(u_{\widetilde\lambda}) \rightarrow t_{\widetilde\lambda}^{n,+}(u_{\widetilde\lambda})=1$ as $j\to\infty$ where $t_{\lambda_j}^{n,+}(u_{\widetilde\lambda})$ and $t_{\widetilde\lambda}^{n,+}(u_{\widetilde\lambda})$ are given by Proposition \ref{compar}. 	
\end{prop}
\begin{proof}
Since $u_{\widetilde\lambda} \in \mathcal{N}_{\widetilde \lambda}^+$ is a weak solution for the our main problem for some $ \widetilde \lambda \in (0, \lambda_n)$ we know that $t_{\widetilde\lambda}^{n,+}(u_{\widetilde\lambda}) =1$. The main idea here is to consider an auxiliary function $F:(0,\lambda_n)\times(0,\infty)\to \mathbb{R}$ defined by 
	$$F(\lambda,t)={R}_n(tu_{\widetilde\lambda})-\lambda, \lambda \in (0, \lambda_n), t > 0.$$
	Notice also that $F \in C^1((0, \lambda_n) \times (0, \infty), \mathbb{R})$. 
	According to Proposition \ref{importante} we observe that 
	\begin{equation}
		F(\widetilde\lambda,1)=0
		\,\,\mbox{and} \,\, \frac{d}{dt}F(\widetilde\lambda,t)_{|_{t=1}}=({R}_n)'(u_{\lambda})(u_{\lambda})>0.
	\end{equation}
Now, we apply the Implicit Function Theorem \cite{drabek} proving that there exist $\delta>0$ and a function in $C^1$ class denoted by $t(.):(\widetilde\lambda-\delta,\widetilde\lambda+\delta)\to\mathbb{R}$ in such way that 
	\begin{itemize}
		\item[(i)] $F(\lambda,t(\lambda))=0, \lambda \in (\widetilde\lambda-\delta,\widetilde\lambda+\delta) $;
		\item[(ii)] $\frac{d}{dt}F(\lambda,t)_{|_{t=t(\lambda)}}=({R}_n)'(t(\lambda)u_{\widetilde\lambda})(t(\lambda)u_{\widetilde\lambda})>0, {\lambda}\in(\widetilde\lambda-\delta,\widetilde\lambda+\delta)$.
	\end{itemize}
	As a consequence, we infer that $t(\lambda)=t_{\lambda}^{n,+}(u_{\widetilde\lambda})$.  In fact, we observe that
	$$F(\lambda,t_{\lambda}^{n,+}(u_{\widetilde\lambda}))={R}_n(t_{\lambda}^{n,+}(u_{\widetilde\lambda})u_{\widetilde\lambda})-\lambda=0, \lambda \in (\widetilde\lambda-\delta,\widetilde\lambda+\delta).$$
	Moreover, we see that 
	\begin{equation}  \frac{d}{dt}F(\lambda,t)_{|_{t=t_{\lambda}^{n,+}(u_{\widetilde\lambda})}}=({R}_n)'(t_{\lambda}^{n,+}(u_{\widetilde\lambda})u_{\widetilde\lambda})(t_{\lambda}^{n,+}(u_{\widetilde\lambda})u_{\widetilde\lambda})>0
	\end{equation}
	The last assertion implies that   $t_{\lambda}^{n,+}(u_{\widetilde\lambda})u_{\widetilde\lambda}\in \mathcal{N}_{\widetilde\lambda}^+$ holds for each  $\lambda\in(\widetilde\lambda-\delta,\widetilde\lambda+\delta)$.
	As a consequence, by using the uniqueness of the projection in the Nehari manifold $\mathcal{N}_{\widetilde\lambda}^+$ given by Proposition \ref{compar},  we deduce that $ t(\lambda) = t_{\lambda}^{n,+}(u_{\widetilde\lambda})u_{\widetilde\lambda}\in \mathcal{N}_{\widetilde\lambda}^+$ holds. As a product, we obtain $t_{\lambda_j}^{n,+}(u_{\widetilde\lambda}) \to t_{\widetilde\lambda}^{n,+}(u_{\widetilde\lambda})=1$ as $j\to\infty$. 	This ends the proof. 
\end{proof}

It is worthwhile to mention that there exists a weak solution $u_{\lambda_j}$ for each $j \in \mathbb{N}$ where $\lambda_j \in (0, \lambda_n)$ satisfies $\lambda_j \rightarrow \widetilde\lambda$. Consider $\mathcal{E}^1_{\lambda_j} = E^1_{\lambda_j}(u_{\lambda_j})$ and $\mathcal{E}^2_{\lambda_j} = E^2_{\lambda_j}(v_{\lambda_j})$ where $u_{\lambda_j}$ and $v_{\lambda_j}$ are minimizers in the Nehari manifolds $\mathcal{N}_{\lambda}^{+}$ and $\mathcal{N}_{\lambda}^{-}$, respectively.
The main feature here is to consider how the weak solutions $u_{\lambda_j}$ and $v_{\lambda_j}$ behave as $j \rightarrow \infty$. In this direction we shall prove the following result:

\begin{prop}\label{Q}
Suppose $(Q)$ and $(V_1)-(V_2)$. Then $(u_{\lambda_j})$ and $(v_{\lambda_j})$ are bounded sequences in $X$.
\end{prop}
\begin{proof}
Firstly, we shall consider the proof for the sequence $(u_{\lambda_j})$. The proof for the sequence $(v_{\lambda_j})$ is analogous. In view of Proposition \ref{F} we observe that
\begin{equation}\label{ed1}
\limsup_{j\to\infty}\mathcal{E}^1_{\lambda_j}\leq\limsup_{j\to\infty} {E}_{\lambda_j}(t_{\lambda_j}^{n,+}(u_{\widetilde\lambda})u_{\widetilde\lambda})={E}_{\widetilde\lambda}(t_{\widetilde\lambda}^{n,+}(u_{\widetilde\lambda})u_{\widetilde\lambda})=\mathcal{E}^1_{\widetilde\lambda}.
\end{equation}
Notice also that $u_{\lambda_j}\in \mathcal{N}_{\lambda_j}$. It follows from the embedding $X\hookrightarrow L^q(\mathbb{R}^N)$ and \eqref{ed1} that 
$$\mathcal{E}^1_{\widetilde\lambda}\geq\limsup_{j\to\infty} E_{\lambda_j}(u_{\lambda_j})\geq \limsup_{j\to\infty}\left[\left(\frac{1}{2}-\frac{1}{2p}\right)\|u_{\lambda_j}\|^2-\lambda_jC\left(\frac{1}{q}-\frac{1}{2p}\right)\|u_{\lambda_j}\|^q\right]$$
holds for some $C>0$. Since $1 < q < 2$ the last inequality says that $(u_{\lambda_j})$ is bounded in $X$. The proof for the sequence $(v_{\lambda_j})$ follows using the same ideas.  This ends the proof. 	
\end{proof}

\begin{prop}\label{convergence}
Suppose $(Q)$ and $(V_1)-(V_2)$. Then, up to a subsequence, there exist $u_{\widetilde\lambda}, v_{\widetilde\lambda} \in X$ in such way that $u_{\lambda_j} \rightarrow u_{\widetilde\lambda}$ and $v_{\lambda_j} \rightarrow v_{\widetilde\lambda}$ in $X$ as $j \rightarrow \infty$.
\end{prop}
\begin{proof}
Recall that $(u_{\lambda_j})$ and $(v_{\lambda_j})$ are bounded sequences in $X$, see Proposition \ref{Q}. As a consequence, up to a subsequence, there exist $u_{\widetilde\lambda}, v_{\widetilde\lambda} \in X$ such that $u_{\lambda_j}\rightharpoonup u_{\widetilde\lambda}$ and $u_{\lambda_j}\rightharpoonup v_{\widetilde\lambda}$ in $X$.
Now we shall consider the proof for the sequence $(u_{\lambda_j})$. The proof for the sequence $(v_{\lambda_j})$ follows using the same ideas. Notice that $(u_{\lambda_j})$ is a sequence of solutions to the following minimization problem   
$$\mathcal{E}^1_{\lambda_j}=\inf_{w\in \mathcal{N}^+_{\lambda_j}}E_{\lambda_j}(w).$$ 
Furthermore, $u_{\lambda_j}$ is a free critical point for the functional $E_{\lambda_j}$ for each $j \in \mathbb{N}$, that is, we know that  $\langle E'_{\lambda_j}(u_{\lambda_j}), \psi \rangle= 0$ for every $\psi \in X$, see Proposition \ref{criticalpoint}. Hence $\langle E'_{\widetilde\lambda}(u_{\widetilde\lambda}), \psi \rangle= 0$ holds for every $\psi \in X$. In particular, we obtain 
\begin{equation}
\langle E'_{\lambda_j}(u_{\lambda_j}),u_{\lambda_j}- u_{\widetilde\lambda} \rangle= 0 \, \,\mbox{and} \,\, \langle E'_{\lambda_j}(u_{\widetilde\lambda}),u_{\lambda_j}- u_{\widetilde\lambda}\rangle= 0.
\end{equation}
As a consequence, we deduce the following estimates 
\begin{equation}\label{ed2}
|\langle u_{\lambda_j},u_{\lambda_j}- u_{\widetilde\lambda}\rangle|\leq \int_{\mathbb{R}^N} \left(\lambda_j|u_{\lambda_j}|^{q-1} +(I_\alpha*|u_{\lambda_j}|^p)|u_{\lambda_j}|^{p-1}\right) \left|u_{\lambda_j}- u\right|=o_j(1).
\end{equation}
and 
\begin{equation}\label{ed3}
|\langle u_{\widetilde\lambda},u_{\lambda_j}- u_{\widetilde\lambda}\rangle|\leq \int_{\mathbb{R}^N} \left(\widetilde\lambda |u_{\widetilde\lambda}|^{q-1} +(I_\alpha*|u_{\widetilde\lambda}|^p)|u_{\widetilde\lambda}|^{p-1}\right) \left|u_{\lambda_j}- u\right| = o_j(1).
\end{equation}
According to \eqref{ed2} and \eqref{ed3} we obtain that $u_{\lambda_j} \rightarrow u$ in $X$. This same argument can be applied for the sequence $(v_{\lambda_j})$.
This finishes the proof. 
\end{proof}

\begin{prop}
Suppose $(Q)$ and $(V_1)-(V_2)$. Then $ u_{\widetilde\lambda} \in \mathcal{N}^+_{\widetilde\lambda}$ and $v_{\widetilde\lambda} \in \mathcal{N}^-_{\widetilde\lambda}$ where $u_{\widetilde\lambda}$ and $v_{\widetilde\lambda}$ was obtained by Proposition \ref{convergence}.
\end{prop}
\begin{proof}
	Recall that $u_{\lambda_j} \rightarrow u_{\widetilde\lambda}$ and $v_{\lambda_j} \rightarrow v_{\widetilde\lambda}$ in $X$ for some $u_{\widetilde\lambda}, v_{\widetilde\lambda} \in X$, see Proposition \ref{convergence}. Under these conditions, by using \eqref{ed1}, we infer that 
	\begin{equation} \displaystyle\lim_{j\to\infty}\mathcal{E}^1_{\lambda_j}=E_{\widetilde\lambda}(u_{\widetilde\lambda})\leq \mathcal{E}^1_{\widetilde\lambda}.
	\end{equation} 
	
	Now we shall split the proof into parts. In the first one we consider the sequence $(u_{\lambda_j})$. In particular, we observe that  
		$$R_n(u_{\lambda_j})=\lambda_j\,\,\mbox{and}\,\, (R_n)'(u_{\lambda_j})(u_{\lambda_j})>0.$$
Taking the limit in the last expression we deduce that 
		$$R_n(u_{\widetilde\lambda})=\widetilde\lambda\,\,\mbox{and}\,\, (R_n)'(u_{\widetilde\lambda})(u_{\widetilde\lambda})\geq0.$$
		Now we claim that $u_{\widetilde\lambda} \neq 0$. In fact, by using \eqref{ed1} and the fact that $E_{\widetilde\lambda}$ is continuous, we infer that 
		$$\displaystyle\lim_{j\to\infty}\mathcal{E}^1_{\lambda_j}=E_{\widetilde\lambda}({u_{\widetilde\lambda}})\leq \mathcal{E}^1_{\widetilde\lambda}<0.$$
		The last assertion ensures that $u_{\widetilde\lambda} \neq 0$ proving the desired claim. Using one more time the strong convergence and due the fact that $E'_{\lambda_j}(u_{\lambda_j}) w = 0, w \in X$ we obtain that $E'_{\widetilde\lambda}(u_{\widetilde\lambda}) w = 0, w \in X$. In other words, $u$ is critical point for $E_{\widetilde\lambda}$ satisfying $u_{\widetilde\lambda} \neq 0$. The last assertion implies that  $u \in \mathcal{N}^+_{\widetilde\lambda}$.

		Now we shall consider the second part involving the sequence $(v_{\lambda_j})$. It is important to point out that the similar argument used in \eqref{ed1} still $\mathcal{E}_\lambda^2$. In this case we obtain that
		\begin{equation}\label{ed4}
		\limsup_{j\to\infty}\mathcal{E}^2_{\lambda_j}\leq\limsup_{j\to\infty} {E}_{\lambda_j}(t_{\lambda_j}^{n,+}(u_{\widetilde\lambda})u_{\widetilde\lambda})={E}_{\widetilde\lambda}(t_{\widetilde\lambda}^{n,+}(u_{\widetilde\lambda})u_{\widetilde\lambda})=\mathcal{E}^2_{\widetilde\lambda}.
		\end{equation}
		Now, by using \eqref{tn} and \eqref{P}, there exists $C > 0$ such that $t_n(u) \geq C \|u\|^{(2 -2p)/(2p -2)}$ holds true for any $u \in X \setminus \{0\}$.
		Furthermore, by using Remark 1.1 and  Proposition \ref{compar}, we mention also that
		\begin{equation}
		C\|v_{\lambda_j}\|^{\frac{2-2p}{2p-2}}\leq t_n(v_{\lambda_j})\leq t_{\lambda_j}^{n,-}(v_{\lambda_j})=1\nonumber.
		\end{equation}
		The last inequality implies that 
		\begin{equation}\label{est-N-}
			\|v_{\lambda_j}\| \geq C > 0
		\end{equation}
		holds true for any $j \in \mathbb{N}$ with $C > 0$. Taking the limit in the expression just above we obtain that $C\leq \|v_{\widetilde\lambda}\|$ is satisfied. In particular, we see also that $v_{\widetilde\lambda} \neq 0$ is now verified.
		This ends the proof. 	
	\end{proof}	
	\begin{prop}\label{monot}
	Suppose $(Q)$ and $(V_1)-(V_2)$. Then we obtain that $\lambda\mapsto\mathcal{E}_\lambda^1$ and $\lambda\mapsto\mathcal{E}_\lambda^2$ are  decreasing functions in $(0,\lambda_n)$.
	\end{prop}	
	\begin{proof}
		Let $\lambda_1,\lambda_2\in(0,\lambda_n)$ be fixed numbers satisfying $\lambda_1<\lambda_2$. Now, we observe that $t_{\lambda_2}^{n,-}(v_{\lambda_1})<t_{\lambda_1}^{n,-}(v_{\lambda_1})=1$, see Figure \ref{mono1}. In fact, we mention that  $$\lambda_1=R_n(t_{\lambda_1}^{n,-}(v_{\lambda_1})v_{\lambda_1})<R_n(t_{\lambda_2}^{n,-}(v_{\lambda_1})v_{\lambda_1})=\lambda_2.$$
		Since $t_{\lambda_2}^{n,-}(v_{\lambda_1}),t_{\lambda_1}^{n,-}(v_{\lambda_1})\in (t_{n}(v_{\lambda_1}),+\infty)$ and $t\mapsto R_n(tv_{\lambda_1})$ is a decreasing  function in $(0,\lambda_n)$ we obtain that $t_{\lambda_2}^{n,-}(v_{\lambda_1})<t_{\lambda_1}^{n,-}(v_{\lambda_1})=1$. 
		
		Now we claim that $t\mapsto E_{\lambda_1}(tv_{\lambda_1})$ is increasing in $[t_{\lambda_1}^{n,+}(v_{\lambda_1}),t_{\lambda_1}^{n,-}(v_{\lambda_1})]$, see Figure \ref{mono2}. The proof for this claim follows using that $t\mapsto E_{\lambda_1}(tv_{\lambda})$ has exactly two critical points given by  $t_{\lambda_1}^{n,+}(v_{\lambda_1})$ and $t_{\lambda_1}^{n,-}(v_{\lambda_1})$. Furthermore, we know that $t_{\lambda_1}^{n,+}(v_{\lambda_1})$  is a local maximum point an $t_{\lambda_1}^{n,-}(v_{\lambda_1})$ is a local minimum point for the fibering function $t \mapsto \phi(t) = E_\lambda(tu)$. In particular, we observe that 
		$E_{\lambda_1}(t_{\lambda_2}^{n,-}(v_{\lambda_1})v_{\lambda_1})<E_{\lambda_1}(t_{\lambda_1}^{n,-}(v_{\lambda_1})v_{\lambda_1})=E_{\lambda_1}(v_{\lambda_1})$. Indeed, by using the fact that  $t_{\lambda_2}^{n,-}(v_{\lambda_1}),t_{\lambda_1}^{n,-}(v_{\lambda_1})\in [t_{\lambda_1}^{n,+}(v_{\lambda_1}),t_{\lambda_1}^{n,-}(v_{\lambda_1})]$ we obtain the desired result. Moreover, using the definition of $\mathcal{E}_{\lambda_1}^2$ together with the fact that $t_{\lambda_2}^{n,-}(v_{\lambda_1}) v_{\lambda_1}$ belongs to $\mathcal{N}_\lambda^-$, we deduce the following estimates 
		\begin{eqnarray}
			\mathcal{E}_{\lambda_2}^2&=&E_{\lambda_2}(v_{\lambda_2})\leq E_{\lambda_2}(t_{\lambda_2}^{n,-}(v_{\lambda_1}) v_{\lambda_1})\nonumber\\
			&=& E_{\lambda_1}(t_{\lambda_2}^{n,-}(v_{\lambda_1}) v_{\lambda_1})-(\lambda_2-\lambda_1)\|t_{\lambda_2}^{n,-}(v_{\lambda_1}) v_{\lambda_1}\|_q^q\nonumber\\
			&<& E_{\lambda_1}(t_{\lambda_2}^{n,-}(v_{\lambda_1}) v_{\lambda_1})<E_{\lambda_1}( v_{\lambda_1})=\mathcal{E}_{\lambda_1}^2.\nonumber
		\end{eqnarray} 
		\begin{figure}[!ht]
			\begin{minipage}[h]{0.49\linewidth}
				\center{\includegraphics[scale=0.7]{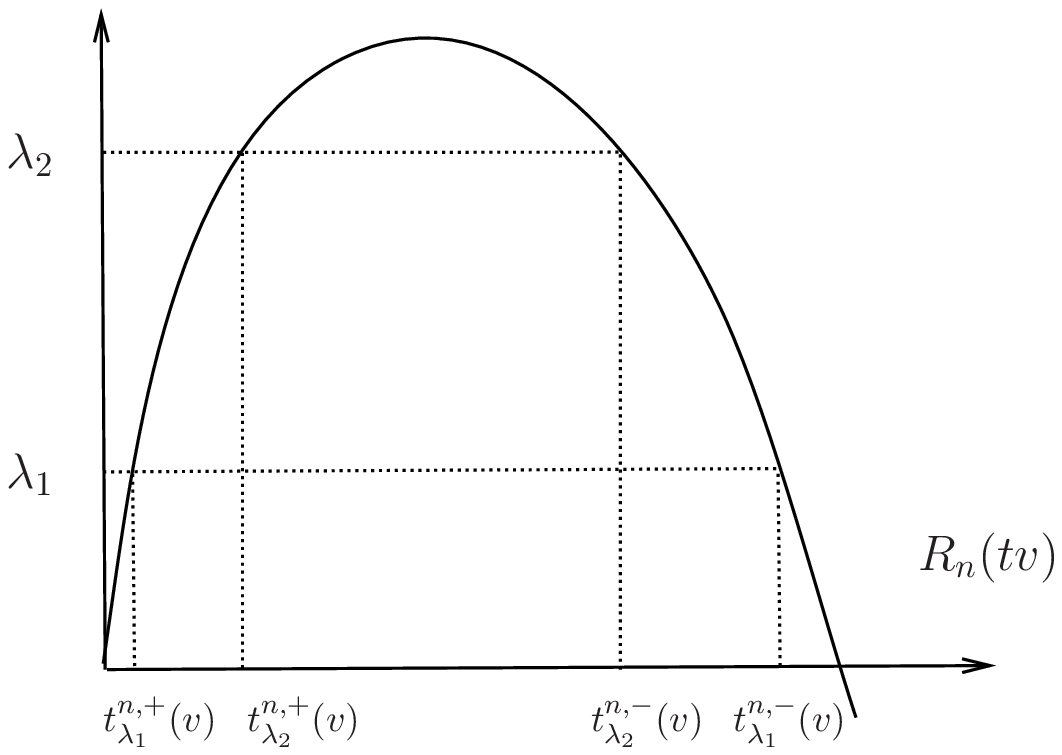}}
				\caption{$\lambda_1<\lambda_2$}
				\label{mono1}
			\end{minipage}
			\begin{minipage}[h]{0.49\linewidth}
				\center{\includegraphics[scale=0.7]{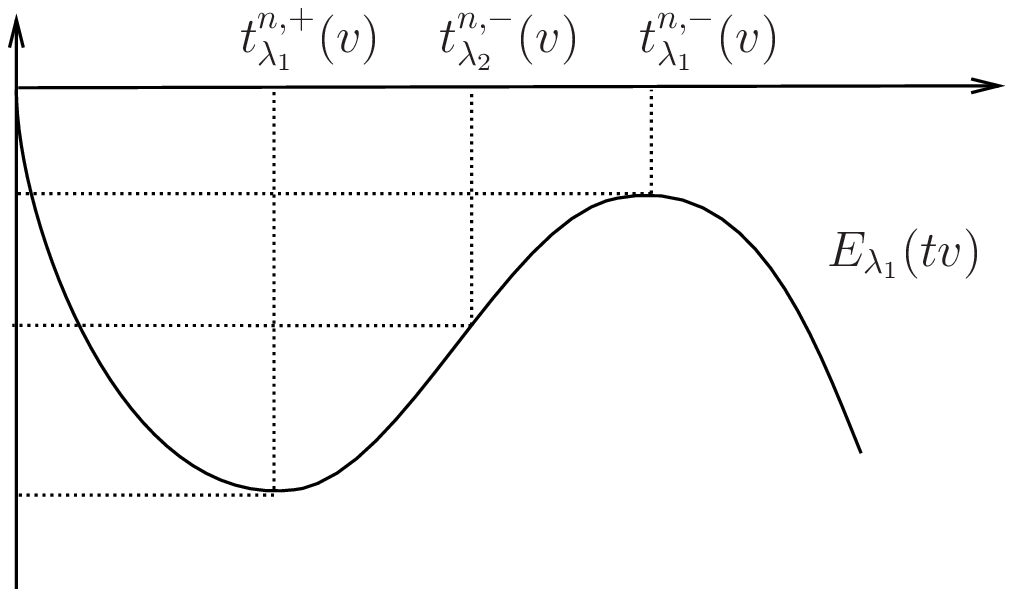}}
				\caption{$\lambda_1<\lambda_2$}
				\label{mono2}
			\end{minipage}
		\end{figure}

	Using the same ideas discussed just above we obtain some analogous properties for the Nehari manifold $\mathcal{N}_\lambda^+$. Namely, we can show the following properties:
		\begin{itemize}
			\item[(i)] $1=t_{\lambda_1}^{n,+}(u_{\lambda_1})<t_{\lambda_2}^{n,+}(u_{\lambda_1})$; 
			\item[(ii)] $t\mapsto E_{\lambda_2}(tu_{\lambda_1})$ is decreasing in $[0,t_{\lambda_2}^{n,+}(u_{\lambda_1})]$; 
			\item[(iii)] $E_{\lambda_2}(t_{\lambda_2}^{n,+}(u_{\lambda_1})u_{\lambda_1})<E_{\lambda_2}(t_{\lambda_1}^{n,+}(u_{\lambda_1})u_{\lambda_1})=E_{\lambda_2}(u_{\lambda_1})$;
			\item[(iv)] $\mathcal{E}_{\lambda_2}^1<\mathcal{E}_{\lambda_1}^1, \lambda_1 < \lambda_2$.
		\end{itemize}
The proof of items $(i)-(iv)$ follows arguing as was done just above. This ends the proof. 
	\end{proof}

\nd{\bf The proof of Theorem \ref{theorem2} completed.} Firstly, by using Proposition \ref{monot}, we know that the functions $\lambda\mapsto\mathcal{E}_\lambda^1$ and $\lambda\mapsto\mathcal{E}_\lambda^2$ are decreasing. Furthermore, we also mention that  
	\begin{eqnarray}\label{compa}
		\mathcal{E}_\lambda^1=E_\lambda(u_\lambda)\leq E_\lambda(t_\lambda^{n,+}(v_\lambda) v_\lambda)<E_\lambda(v_\lambda)=\mathcal{E}_\lambda^2.
	\end{eqnarray}
	These facts ensure that the item $i)$ is now verified.


Now we shall prove the item $ii)$. It follows from the previous item and the definition of $\mathcal{E}^1_{\widetilde\lambda}$ that 
	$E_{\widetilde{\lambda}}(u)=\mathcal{E}^1_{\widetilde\lambda}:=E_{\widetilde{\lambda}}(u_{\widetilde\lambda})$ holds true for some $u \in \mathcal{N}_{\widetilde\lambda}^+$.
	The last assertion implies that $u= u_{\widetilde\lambda}$. Here was used the fact that  $u,u_{\widetilde\lambda} \in \mathcal{N}_{\widetilde{\lambda}}^+$, see Proposition \ref{compar}. In particular, we obtain that $u_{\lambda_j} \rightarrow u_{\widetilde\lambda}$ as $j \rightarrow \infty$. The same argument can be applied for the sequence $v_{\lambda_j}$ proving that $v_{\lambda_j} \rightarrow v_{\widetilde\lambda}$ as $j \rightarrow \infty$. This ends the proof of item $ii)$.
	
	Now we shall prove that $\lambda\mapsto\mathcal{E}_\lambda^{i},~\lambda\in(0,\lambda_n),~i=1,2,$ is continuous. Indeed, define the function $E:(0, + \infty)\times X\rightarrow \mathbb{R}$ given by $E(\lambda,u):=E_\lambda(u)$. It is easy ensure that $E$ is a continuous function. In particular, given any sequence $\lambda_j\subset (0,\lambda_n)$ such that $\lambda_j\to \lambda\in(0,\lambda_n)$ we infer that $u_{\lambda_j}\to u_\lambda$, see Proposition \ref{convergence}. Therefore, we obtain 
		$\mathcal{E}_{\lambda_j}^1=E(\lambda_j,u_{\lambda_j})\to E(\lambda,u_{\lambda})=\mathcal{E}_{\lambda}^1.$
		The last statement implies that $\lambda\mapsto\mathcal{E}_\lambda^1$ is a continuous function for each $\lambda\in(0,\lambda_n)$. The same argument can be applied for the function $\mathcal{E}^2_{\widetilde\lambda}$ showing that item $iii)$ is now verified. This finishes the proof.  
	
\subsection{Proof of Theorem \ref{theorem3}}
In this subsection we shall prove Theorem \ref{theorem3}. Firstly, we consider some auxiliary results. In the first one we consider the following result:  
\begin{lem}\label{vanish}
	Suppose $(Q)$ and $(V_1)-(V_2)$. Then $\mathcal{N}_0^+$ and $\mathcal{N}_0^0$ are empty sets.
\end{lem}
\begin{proof}
	The proof follows arguing by contradiction. Assume that there exists $u_0\in \mathcal{N}_0^0$. Hence the function $u_0$ satisfies the following identities 
	$$E'_0(u_0)(u_0)=0 \quad \mbox{and} \quad E''_0(u_0)(u_0,u_0)=0, u_0 \neq 0.$$
Now, solving the system just above in the variables $\|u_0\|^2$ and {$\int_{\Omega}(I_\alpha*|u_0|^p)|u_0|^pdx$} and taking into account \eqref{segunda}, we obtain the following identities  
		$$\|u_0\|^2=\displaystyle\int_{\Omega}(I_\alpha*|u_0|^p)|u_0|^pdx=0.$$
As a consequence, $u_0\equiv 0$ which is absurd. The last assertion implies that $\mathcal{N}_0^0=\emptyset$. The proof for $\mathcal{N}_0^+$ follows the same ideas discussed just above. This ends the proof.	
\end{proof}	

\begin{prop}\label{convforte}
	Suppose $(Q)$ and $(V_1)-(V_2)$. Let $(\lambda_j)\subset (0,\lambda_n)$ be a sequence such that $\lambda_j\to 0$. Then $u_{\lambda_j} \to 0$ and $v_{\lambda_j} \to v_0$ as $j\to\infty$ where $v_0$ is a weak solution for the Problem \eqref{eq1} with $\lambda=0$.
\end{prop}
\begin{proof}
	Initially, we shall consider the sequence $(u_{\lambda_j})$. Remember that
$\mathcal{E}^1_{\lambda_j}=E_{\lambda_j}(u_{\lambda_j})$ and $u_{\lambda_j}$ satisfies the following conditions:
\begin{itemize}
	\item[(i)] $u_{\lambda_j}$ is a weak solution for Problem \eqref{eq1} with $\lambda=\lambda_j$;
	\item[(ii)] $u_{\lambda_j}\in \mathcal{N}_{\lambda_j}^+$ and $(u_{\lambda_j})$ is a minimizer in $\mathcal{N}_{\lambda_j}^+$.
\end{itemize}
	Now we claim that $(u_{\lambda_j})$ is bounded. In fact, using that $u_{\lambda_j}\in \mathcal{N}_{\lambda_j}$ and Lemma \ref{sing-u} together with  the compact embedding $X\hookrightarrow L^r(\mathbb{R}^N)$ for each $r\in [1,2^*)$, we obtain
	$$\left(\frac{1}{2}-\frac{1}{2p}\right)\|u_{\lambda_j}\|^2-\lambda_j\left(\frac{1}{q}-\frac{1}{2p}\right)C\|u_{\lambda_j}\|^q_q\leq \mathcal{E}_{\lambda_j}^1<0.$$
	Hence the sequence $(u_{\lambda_j})$ is now bounded. Moreover, using the same ideas discussed in the proof of Proposition \ref{convergence}, we also deduce that $u_{\lambda_j}\to u_0$ in $X$. It follows from  items $(i)$ and $(ii)$ given just above that
	\begin{itemize}
		\item[(a)] $u_0$ is a solution of \eqref{eq1} with $\lambda=0$;
		\item[(b)] $u_0$ satisfies $E'_0(u_0)(u_0)=0$ and $E''_0(u_0)(u_0,u_0)\geq0$.		
	\end{itemize}
	According to item $(b)$ we observe that $u_0\in \mathcal{N}_0\cup \{0\}$. In view of Lemma \ref{vanish} the set $\mathcal{N}_0^+=\emptyset$. This assertion implies that $E''_0(u_0)(u_0,u_0)=0$. As a consequence, applying Lemma \ref{vanish} once more, we deduce $u_0\equiv 0$. The last assertion says that  $u_{\lambda_j} \to 0$ as $j\to\infty$.	
	
	Now we shall consider the sequence $(v_{\lambda_j})$. Notice that $\mathcal{E}^2_{\lambda_j}=E_{\lambda_j}(v_{\lambda_j})$
	and $v_{\lambda_j}$ satisfies the following properties: 
	\begin{itemize}
		\item[(iii)] $v_{\lambda_j}$ is a nontrivial weak solution for the Problem \eqref{eq1} with $\lambda=\lambda_j$;
		\item[(iv)] $v_{\lambda_j}\in \mathcal{N}_{\lambda_j}^-$ and $(u_{\lambda_j})$ is a minimizer in $\mathcal{N}_{\lambda_j}^-$.
	\end{itemize}
	Note that $(v_{\lambda_j})$ is also bounded sequence in $X$. In fact, we observe that  $\mathcal{E}^2_{\lambda_j}=E_{\lambda_j}(v_{\lambda_j})\leq E_0(u)$ holds for all $u\in X$. In particular, we infer that 
	$$\frac12\|v_{\lambda_j}\|^2-\frac{1}{2p}\int_{\mathbb{R}^N}(I_\alpha*|v_{\lambda_j}|^p)|v_{\lambda_j}|^{p}dx-\frac1q\|v_{\lambda_j}\|_q^q\leq \mathcal{E}^2_{0}.$$
	Furthermore, using that $v_{\lambda_j}\in\mathcal{N}_{\lambda_j}$ and $\lambda_j\in(0,\lambda_n)$, one has
	$$\left(\frac{1}{2}-\frac{1}{2p}\right)\|v_{\lambda_j}\|^2-\lambda_n\left(\frac{1}{q}-\frac{1}{2p}\right)C\|v_{\lambda_j}\|^q_q\leq \mathcal{E}_{0}^2.$$
	Hence $(v_{\lambda_j})$ is now bounded sequence in $X$. Using the ideas employed just above we deduce that $v_{\lambda_j}\to v_0$ in $X$ as $j\to \infty$. Moreover, by using the same ideas discussed in the proof of \eqref{est-N-}, we conclude that there exists a constant $C>0$ which does not depend on $\lambda$ such that $C \leq \|v_{\lambda_j}\|$. The last statement and the strong converge given just above show that $v_0\neq 0$. Under these conditions, taking the limit $j\to \infty$, it follows from items $(iii)$ and $(iv)$ the following statements
	\begin{itemize}
		\item[(c)] $v_0$ is a weak solution for the Problem \eqref{eq1} with $\lambda=0$;
		\item[(d)] $v_0$ satisfies $E_0'(v_0)(v_0)=0$ and $E''_0(v_0)(v_0,v_0)\leq0$.		
	\end{itemize} 
	Now we shall argue by contradiction. Assume that the function $v_0$ satisfies 
	$E''_0(v_0)(v_0,v_0)=0.$
	As a consequence, by using the same strategy employed in the proof of Lemma \ref{vanish}, we also see that  $v_0\equiv 0$.  This is a contradiction proving that the function $v_0$ satisfies the following conditions
	$$E'_0(v_0)(v_0)=0\quad\mbox{and}\quad E''_0(v_0)(v_0,v_0) < 0.$$
	Hence $v_0 \in \mathcal{N}_0^{-}$. It follows from item $(c)$ that $v_0$ is a weak solution of Problem \eqref{eq1}. This ends the proof. 	
\end{proof}

\subsection{The proof of Theorem \ref{theorem4}}
In this subsection the main objective is to find existence of weak solutions for Problem \eqref{eq1} assuming that $\lambda= \lambda_n$. In order to do that in the present section we shall consider a sequence $\lambda_j\in(\lambda_e,\lambda_n)$ such that $\lambda_j\uparrow \lambda_n$. It follows from Theorem \ref{theorm1} that there exist weak solutions $u_{\lambda_j}$ and $v_{\lambda_j}$ for Problem \eqref{eq1} for each $j \in \mathbb{N}$. Now, we consider some auxiliary results describing how the Nehari manifold behaves assuming that $\lambda= \lambda_n$. Initially, we consider the following result: 
\begin{prop}
	Suppose $(Q)$ and $(V_1)-(V_2)$. Then $\mathcal{N}_{\lambda_n}^0\neq\emptyset$.
\end{prop}
\begin{proof}
	In view of Lemma \ref{Lambda} there exists $u_{\lambda_n}\in X\setminus\{0\}$ such that
	$$\lambda_n=\Lambda_n(u_{\lambda_n })=\inf_{u\in X\setminus\{0\}}\Lambda_n(u)=R_n(u_{\lambda_n}).$$
	The last equation ensures that $u_{\lambda_n}\in \mathcal{N}_{\lambda_n}$. Furthermore,  by using the fact that $t=1$ is a maximum point for $Q_n(t):=R_n(tu_{\lambda_n})$, we observe that
	\begin{eqnarray}\label{equiv}
		0=Q_n'(1)=(R_n)'(u_{\lambda_n})u_{\lambda_n}= \|u_{\lambda_n}\|_q^{-q} E_{\lambda_n}''(u_{\lambda_n})(u_{\lambda_n},u_{\lambda_n}).
	\end{eqnarray}
The last assertion says that $u_{\lambda_n}\in \mathcal{N}_{\lambda_n}^0$. This ends the proof. 
\end{proof}

\begin{prop}\label{salva}
Suppose $(Q)$ and $(V_1)-(V_2)$. Then does not exist weak solutions $u \in X$ for the Problem \eqref{eq1} such that $u \in \mathcal{N}_{\lambda_n}^0$.	
\end{prop}
\begin{proof}
	The proof follows arguing by contradiction. Assume that there exist $u \in \mathcal{N}_{\lambda_n}^0$, i.e., $E_{\lambda_n}''(u)(u,u)=0$ in such way that $u$ is a weak solution for our main problem \eqref{eq1}. Hence $E'_{\lambda_n}(u) \psi = 0$ for every $\psi \in X$. Here we observe that $u$ satisfies $\lambda_n=\Lambda_n(u)$. As a consequence, the function $u$ is also a weak solution of \eqref{eq-lambda-n}. Under these conditions, we obtain
\begin{eqnarray}
(I_\alpha*|u|^p)|u|^{p-2}u=\frac{(2-q)\lambda_n}{2p-2}|u|^{q-2}u,~
\mbox{a.e. in }\mathbb{R}^N.\nonumber
\end{eqnarray}
The last assertion implies that 
\begin{eqnarray}
(I_\alpha*|u|^p)|u|^{p-q}=\frac{(2-q)\lambda_n}{2p-2},~\mbox{a.e. in }[u\neq 0].\nonumber
\end{eqnarray}
Notice also that $u \in  C^{1, \beta}_{loc}(\mathbb{R}^N)$ for some $\beta \in (0, 1)$
which can be done using Theorem \ref{regular} in Appendix. In particular, the last assertion implies
also that $u \in C^{0}(\mathbb{R}^N)$ showing that $u (x) \to 0$ as $|x| \to \infty$. In this case, by using the strong maximum principle and the same ideas
discussed in the proof of Proposition \ref{exx}, we mention that $u > 0$ in $\mathbb{R}^N$. Furthermore, by using Lemma \ref{esti-I} in
Appendix, there exists constants $C, R > 0$ such that
$$\left|\frac{(I_\alpha*|u|^p)(x)}{I_\alpha(x)}\right|\leq C\|u\|_p^p, \quad I_\alpha(x)<\frac{1}{C\|u\|_p^p}, x \in B_R(0)^c.$$
Notice also that $|u|^{p-q}<\frac{(2-q)\lambda_n}{2p-2}$ in $B_R(0)^c$ holds for any $R > 0$ large enough. As a product
\begin{eqnarray}
\frac{(2-q)\lambda_n}{2p-2}=\frac{(I_\alpha*|u|^p)(x)}{I_\alpha(x)}I_\alpha(x)|u(x)|^{p-q}(x)<\frac{(2-q)\lambda_n}{2p-2},~\mbox{a.e. in }B_R(0)^c.\nonumber
\end{eqnarray}
This is a contradiction proving that any weak solution $u$ for the Problem \eqref{eq1} satisfies $u \notin \mathcal{N}_{\lambda_n}^0$. This ends the proof.
\end{proof}

	 \begin{prop}\label{minimizer} 	Suppose $(Q)$ and $(V_1)-(V_2)$.	
				Let $(v_k)$ be a minimizer sequence for $E_{\lambda_n}$ restricted to $\mathcal{N}^{-}_{\lambda_n}$. Then there there exists $v_{\lambda_n} \in X$
				such that $v_k \to v_{\lambda_n}$ in $X$.
		\end{prop}
\begin{proof}
	Initially, we observe that there exists $\psi \in\mathcal{N}_{\lambda_n}^-$ such that $E_{\lambda_n}(\psi)<0$. In fact, by using the definition of $\lambda_e$ there exists $u\in X$ such that $\lambda_e<\Lambda_e(u)<\lambda_n$. Here we also mention that $\Lambda_n(u)>\lambda_n$. Otherwise, we must obtain $\Lambda_n(u)=\lambda_n$. Taking into account Remark \ref{impor} we also obtain $\Lambda_e(u)=\lambda_e$ which is a contradiction with $\lambda_e < \Lambda_e(u)$. In this way, $\Lambda_n(u)>\lambda_n$ is now verified proving that $t_{\lambda_n}^{n,-}(u)$ is well defined. Furthermore, we observe that $\Lambda_e(u)<\lambda_n$. As a consequence, we deduce that $t_{\lambda_n}^{n,-}(u)<t_e(u)$, see  Figure \ref{QnQe-maior-lambda-e} with $\lambda=\lambda_n$. Once more using the fact that ${R}_n(tu)>{R}_e(tu)$ for any $t\in(0,t_e(u))$ the last inequality says that 
		$$\lambda_n={R}_n(t_{\lambda_n}^{n,-}(u)u)>{R}_e(t_{\lambda_n}^{n,-}(u)u).$$
	Therefore, by applying Remark \ref{rmk11}-iii), we conclude that $E_{\lambda_n}(t_{\lambda_n}^{n,-}(u)u)<0$. For our purposes we define $\psi=t_{\lambda_n}^{n,-}(u)u \in \mathcal{N}_\lambda^-$.
	
	Let $(v_k)$ be a minimizer sequence for $E_{\lambda_n}{{\Large{|}}_{\mathcal{N}^-_{\lambda_n}}}$. Since  $E_{\lambda_n}{\Large{|}}_{{\mathcal{N}^-_{\lambda_n}}}$ is coercive we observe that $(v_k)$ is bounded in $X$. Hence, up to a subsequence,  there exists $v_{\lambda_n}\in X$ such that $v_k\rightharpoonup v_{\lambda_n}$ in $X$. Notice also that the functional $u \mapsto E_{\lambda_n}(u)$ is weakly lower semicontinuous. As a consequence,  we obtain
	$$E_{\lambda_n}(v_{\lambda_n})\leq\liminf E_{\lambda_n}(v_k) = \mathcal{E}_{\lambda_n}^2 \leq E_{\lambda_n}(\psi)< 0$$
	holds true for $\psi \in \mathcal{N}_{\lambda_n}^-$ defined above. In particular, the last estimate ensures that $v_{\lambda_n}\neq 0$. From now on the proof follows arguing by contradiction. Assume that $(v_k)$ does not converge to $v_{\lambda_n}$ in $X$. The last assertion implies that
\begin{eqnarray}
	E_{\lambda_n}(v_{\lambda_n})&<&\liminf E_{\lambda_n}(v_k)\label{e1}\\
	R_{n}(v_{\lambda_n})&<&\liminf R_{n}(v_k)\label{e2}=\lambda_n\\
	R'_{n}(v_{\lambda_n})v_{\lambda_n}&<&\liminf R'_{n}(v_k)v_k\leq 0\label{e3}.
	\end{eqnarray}
In view of \eqref{e2} we infer that $E'_{\lambda_n}(v_{\lambda_n})(v_{\lambda_n}) < 0$, see Remark \ref{rmk1}. Recall also that $E''_{\lambda}(v_k)(v_k, v_k) < 0$ which together with \eqref{e3} implies that $E''_{\lambda}(v_{\lambda_n})(v_{\lambda_n}, v_{\lambda_n}) < 0$, see Proposition \ref{der-Rn}.  As a product we obtain that $t^{n,-}_{\lambda_n}(v_{\lambda_n}) < 1 = t_{\lambda_n}^{n,-}(v_{k})$ holds for any $k \in \mathbb{N}$.
Since $(v_k)$ does not converge to $v_{\lambda_n}$ in $X$  we deduce the following estimates 
		$$
		\mathcal{E}_{\lambda_n}^2 \leq E_{\lambda_n}(t^{n,-}_{\lambda_n}(v_{\lambda_n})v_{\lambda_n}) < \liminf E_{\lambda_n}(t^{n,-}_{\lambda_n}(v_{\lambda_n})v_k)\leq \liminf E_{\lambda_n}(v_k)=\mathcal{E}_{\lambda_n}^2.$$ 
		This is a contradiction proving that $v_k\to v_{\lambda_n}$. This ends the proof. 
	\end{proof}

Now, recalling that  $(u_{\lambda_j})$ and $(v_{\lambda_j})$  are weak solutions for Problem \eqref{eq1} where $\lambda_j\in(\lambda_e,\lambda_n)$ is a sequence in such way that $\lambda_j\uparrow \lambda_n$. Hence we can prove the following result:
\begin{prop}\label{incial}
	Suppose $(Q)$ and $(V_1)-(V_2)$.  Assume also that $\lambda_j\uparrow \lambda_n$ as $j \to \infty$. Then, up to subsequences, there exist $u_{\lambda_n}\in X$ and $v_{\lambda_n}\in X\setminus\{0\}$ such that
	\begin{itemize}
		\item[(i)] $u_{\lambda_j}\to u_{\lambda_n}$ and $v_{\lambda_j}\to v_{\lambda_n}$ in $X$;
		\item[(ii)] $u_{\lambda_n}\in \mathcal{N}_{\lambda_n}^+$ and $v_{\lambda_n}\in \mathcal{N}_{\lambda_n}^-$;
		\item[(iii)] $u_{\lambda_n}$ and $v_{\lambda_n}$ are distinct weak solutions of Problem \eqref{eq1} where $\lambda=\lambda_n$;
		\item[(iv)] $\mathcal{E}_{\lambda_n}^1<\mathcal{E}_{\lambda_n}^2$.
	\end{itemize}
\end{prop}

\begin{proof}
Initially we shall prove the item $(i)$. Using the same ideas discussed in the proof of Proposition \ref{convforte} we infer that $(u_{\lambda_j})$ and $(v_{\lambda_j})$ are bounded sequences. Now, using once more the same ideas employed in the proof of Proposition \ref{convergence}, we see that $u_{\lambda_j}\to u_{\lambda_n}$ and $v_{\lambda_j}\to v_{\lambda_n}$ in $X$. 
	
Now we shall prove the item $(ii)$. Notice that $E'_{\lambda_j}(u_{\lambda_j})u_{\lambda_j}=E'_{\lambda_j}(v_{\lambda_j})v_{\lambda_j}=0$ and $u_{\lambda_j}, v_{\lambda_j}$ are critical points for the functional $E_{\lambda_j}$. Furthermore, we also observe that
	$$E''_{\lambda_j}(u_{\lambda_j})(u_{\lambda_j},u_{\lambda_j})>0\qquad\mbox{and}\qquad E''_{\lambda_j}(v_{\lambda_j})(v_{\lambda_j},v_{\lambda_j})<0.$$
	At this stage, taking the limit and using the strong convergence quoted in the item $(i)$, we deduce that $E'_{\lambda_n}(u_{\lambda_n})u_{\lambda_n}=E'_{\lambda_n}(_{\lambda_n})v_{\lambda_n}=0$. In the same way, we also obtain 
	$$E''_{\lambda_n}(u_{\lambda_n})(u_{\lambda_n},u_{\lambda_n})\geq0\quad\mbox{and}\quad E''_{\lambda_n}(v_{\lambda_n})(v_{\lambda_n},v_{\lambda_n})\leq0.$$
	Moreover, by using Lemma \ref{E-}, we also obtain $v_{\lambda_n}\neq 0$. The last assertion says that $v_{\lambda_n}\in \mathcal{N}_{\lambda_n}^-\cup \mathcal{N}_{\lambda_n}^0$. However, there is no any weak solution $v \in X$ for the Problem \eqref{eq1} such that $v \in \mathcal{N}_\lambda^0$, see Proposition \ref{salva}. As a product we know that $v_{\lambda_n}$ is in $\mathcal{N}_{\lambda_n}^-$. According to Proposition \ref{monot} we know that $(\mathcal{E}_{\lambda_j}^1)$ is a decreasing sequence. Thus, using that $u_{\lambda_j}\to u_{\lambda_n}$ in $X$, we infer also that
	$$E_{\lambda_n}(u_{\lambda_n})=\lim_{j\to\infty} E_{\lambda_j}(u_{\lambda_n})< E_{\lambda_j}(u_{\lambda_j})=\mathcal{E}_{\lambda_j}^1<0,~\forall j\in\mathbb{N}.$$
	Hence $u_{\lambda_n}\neq 0$ showing that $u_{\lambda_n}\in \mathcal{N}_{\lambda_n}^+\cup \mathcal{N}_{\lambda_n}^0$ holds true. Using the same ideas discussed just above we also prove that $u_{\lambda_n}\in \mathcal{N}_{\lambda_n}^+$.	This ends the proof of item $(ii)$.	
	
	The proof of item $(iii)$ follows immediately from items $(i)$ and $(ii)$. The proof of item $(iv)$ follows from Proposition \ref{minimizer}. In fact, there exists $v_{\lambda_n} \in \mathcal{N}_\lambda^+$ such that $\mathcal{E}_{\lambda_n}^2=E_{\lambda_n}(v_{\lambda_n})$. Notice also that $t_{\lambda_n}^{n,+}(v_{\lambda_n}) > 0$ is well defined, see Proposition \ref{compar}. As a consequence, we mention that 
	$$\mathcal{E}_{\lambda_n}^1\leq {E}_{\lambda_n}(t_{\lambda_n}^{n,+}(v_{\lambda_n})v_{\lambda_n})<{E}_{\lambda_n}(v_{\lambda_n})=\mathcal{E}_{\lambda_n}^2.$$
This ends the proof.
\end{proof}

\nd{\bf The proof of Theorem \ref{theorem4} completed.}  As was mentioned before we consider a sequence $\lambda_j\in(\lambda_e,\lambda_n)$ such that $\lambda_j\uparrow \lambda_n$. It follows from Theorem \ref{theorm1} that there exist weak solutions $u_{\lambda_j}$ and $v_{\lambda_j}$ for Problem \eqref{eq1} for each $j \in \mathbb{N}$. Hence $u_{\lambda_j} \to u_{\lambda_n}$ and $v_{\lambda_j} \to v_{\lambda_n}$ in $X$, see Proposition \ref{incial}. The last assertion implies also that $u_{\lambda_n}$ and $v_{\lambda_n}$ are weak solutions for the Problem \eqref{eq1}. Furthermore, using the same ideas employed in the proof of Proposition \ref{exx}, we observe that $u_{\lambda_n} > 0$ and $v_{\lambda_n} > 0$ in $\mathbb{R}^N$.

\section{Appendix}
In this Appendix we collect some useful results for nonlocal elliptic problems involving Choquard equation. The next result was motivated by \cite[Lemma 6.2]{Muroz1}.
\begin{lem}\label{esti-I}
	Assume that ${N \geq 3}, \alpha \in (0, N)$ and $2_\alpha < p < 2^*_\alpha$ hold. Then there exists a constant $C>0$ such that
	$$\left|\frac{(I_\alpha*|u|^p)(x)}{I_\alpha(x)}-\|u\|_p^p\right|\leq C\|u\|_p^p, u \in X.$$
\end{lem}
\begin{proof}
	Initially, we observe that there exists a constant $C>0$ such that
	$$\left|\frac{1}{|x-y|^{N-\alpha}}-\frac{1}{|x|^{N-\alpha}}\right|\leq \frac{C}{|x|^{N-\alpha}}.$$
	holds for all $x,y\in\mathbb{R}^N$, see \cite[Lemma 6.2]{Muroz1}.
	As a consequence, we obtain 
	\begin{eqnarray}
		\left|(I_\alpha*|u|^p)(x)-I_\alpha(x)\|u\|_p^p\right|&\leq&\Int |u|^p\left|\frac{1}{|x-y|^{N-\alpha}}-\frac{1}{|x|^{N-\alpha}}\right| dx \leq  C\frac{\|u\|_p^p}{|x|^{N-\alpha}}, u \in X.\nonumber
	\end{eqnarray}
This ends the proof.
\end{proof}
At this stage, we shall prove a regularity result for our main problem. More specifically, we consider the following result:
\begin{theorem} \label{regular} Suppose $V\in L^{\frac{N}{2}}_{loc}(\mathbb{R^N})$ and $1<q \leq 2^*$ and $p \in (2_\alpha, 2_\alpha^*)$.  Assume that $u\in X$ is a weak solution of Problem \eqref{eq1}. Then $u\in W^{2,r}_{loc}(\mathbb{R}^N)\cap C_{loc}^{1,\beta}(\mathbb{R}^N)$ for all $r \in (1, \infty)$ and for some $\beta \in (0, 1)$.
	\end{theorem}
\begin{proof} Let $u\in X $ be a weak solution of Problem \eqref{eq1}, that is, we have that
$$	-\Delta u+V(x)u=(I_\alpha*|u|^p)|u|^{p-2}u+\lambda|u|^{q-2}u,\ \  in\ \  \mathbb{R}^N.$$
Here we define 
\begin{equation}
a(x)=
\dfrac{g(x,u)}{1+|u|}.
\end{equation}
where
$g(x,u)= (I_\alpha*|u|^p)|u|^{p-2}u+\lambda|u|^{q-2}u-V(x)u, x \in \mathbb{R}^N$.
Now we rewrite the nonlocal problem given just above as follows
$$-\Delta u=a(x)(1+|u|), x \in \mathbb{R}^N.$$ 
It is not hard to see that 
\begin{eqnarray}\label{ee}
\lambda |u|^{q-1}&\leq& \left\{\begin{array}{rcl}\lambda(1+|u|), \, \mbox{for} \, \, q\in(1,2),\\[1ex]
\lambda|u|^{q-2}(1+|u|), \, \, \mbox{for} \,\, q\in [2,2^*],
\end{array}\right.\\[1ex]
V(x)|u|&\leq& V(x)(1+|u|),\\[1ex]
(I_\alpha*|u|^p)|u|^{p-1}&\leq& (I_\alpha*|u|^p)|u|^{p-2} (1+|u|).
\end{eqnarray}
As a consequence, we also see that 
\begin{equation}\label{ei}
|a(x)|^{\frac{N}{2}}\leq \left\{\begin{array}{rcl}
&[(I_\alpha*|u|^p)|u|^{p-2}+\lambda+V(x)]^{\frac{N}{2}},\, \, \mbox{for} \, \,  q \in (1,2),\\[1ex]
&[(I_\alpha*|u|^p)|u|^{p-2}+\lambda|u|^{q-2}+V(x)]^{\frac{N}{2}},\, \, \mbox{for} \, \,  q \in [2,2^*].
\end{array}
\right.
\end{equation}
Now we shall split the proof into three cases. In the first one we assume that that $1 \leq (q-2)N/2$. Notice that $(q-2)N/2 \leq 2^*$ where $2^* = 2 N /(N -2)$ for each $N \geq 3$. As a consequence, by using the Sobolev embedding $H^1(\mathbb{R}^N)\hookrightarrow L_{loc}^r(\mathbb{R}^N), r \in [1, 2^*], 2^* = 2 N/(N - 2)$, we deduce that $\lambda|u|^{q-2}\in L_{loc}^{N/2}(\mathbb{R}^N)$. In the second case we assume that 
$0 \leq (q-2)\frac{N}{2} < 1$. For this case we observe that  
\begin{eqnarray}\label{eee}
\lambda |u|^{(q-2)N/2}&\leq& \left\{\begin{array}{rcl} \lambda, \, \mbox{for} \, \, |u| \leq 1,\\[1ex]
\lambda |u|, \, \, \mbox{for} \,\, |u| \geq 1
\end{array}\right.
\end{eqnarray}
Therefore, by using the last estimate, we ensure that $\lambda|u|^{q-2}\in L_{loc}^{N/2}(\mathbb{R}^N)$ for each $q$ satisfying $0 \leq (q-2)\frac{N}{2} < 1$. It remains to consider the third case where  $1 < q < 2$ holds true. For this case, by using \eqref{ee}, we mention that $\lambda|u|^{q-1} \leq \lambda (1 + |u|)$.  Furthermore, we claim that 
$\ (I_\alpha*|u|^p)|u|^{p-2} \in L_{loc}^{N/2}(\mathbb{R}^N)$ holds. Assuming the claim given just above we infer that $a \in L_{loc}^{N/2}(\mathbb{R}^N)$. This can be done using the claim together with \eqref{ei} and \eqref{eee}. 
In view of \cite[Lemma B3]{struwe} we conclude that $u \in W_{loc}^{2,r}(\mathbb{R}^N)$ holds
true for any $r >  1$. Here was used the Calderon-Zygmund Theorem, see \cite[Theorem 9.9]{GT}. Now, by using the Sobolev embedding, we infer that $u\in C^{1,\beta}_{loc}(\mathbb{R}^N)$ for some $\beta \in (0,1)$.

It remains to prove that $\ (I_\alpha*|u|^p)|u|^{p-2} \in L_{loc}^{N/2}(\mathbb{R}^N)$. In order to do that we define $b: \mathbb{R}^N \to \mathbb{R}$ given by
\begin{equation}\label{b}
	b(x)=\left\{\begin{array}{rcl}&\Int(I_\alpha*|u|^p)|u|^{p-2} dx,&\, \, \mbox{for} \, \, u\neq 0,\\
&0&\, \,  \mbox{for} \,\, u=0.
\end{array}
\right.
\end{equation}
Hence, using the H\"older inequality, for any $B_R(0) \subset\mathbb{R}^N, R > 0,$ we obtain
$$
\displaystyle\int_{B_R(0)}|b(x)|^\frac{N}{2}dx\leq \left(\displaystyle\int_{B_R(0)}\left(I_\alpha*|u|^p\right)^{\frac{N}{2}r'}dx\right)^{\frac{1}{r'}}||u||_{2^*}^{\frac{(p -2)N}{2 \cdot 2^*}},\ \ r=\dfrac{4}{(N-2)(p-2)}> 1.
$$
The last estimate can be rewritten in the following form
$$
\displaystyle\int_{B_R(0)}|b(x)|^\frac{N}{2}dx\leq \left(\displaystyle\int_{B_R(0)}\left(I_\alpha*|u|^p\right)^{\frac{2*}{2^*-p}}dx\right)^{\frac{1}{r'}}||u||_{2^*}^{\frac{2^*}{2^*-p}}.
$$
For our purpose is enough to ensure that
$$
 \displaystyle\int_{B_R(0)}\left((I_\alpha*|u|^p)\right)^{\frac{2*}{2^*-p}}dx<+\infty
$$
In view of Lemma \ref{lemaHardy} the last integral is finite provided that $|u|^p\in L^s(\mathbb{R^N})$ where $s\in (1, N/\alpha)$ satisfies the following condition $\frac{Ns}{N-\alpha s}=\frac{2^*}{2^*-p}$. The last identity can be written in the following form $s=\frac{2N}{4+2\alpha-(N-2)(p-2)}.$ Under these conditions we infer that 
$$\displaystyle\int_{B_R(0)}|u|^{ps}dx=\displaystyle\int_{B_R(0)}|u|^{\frac{2Np}{4+2\alpha-(N-2)(p-2)}}dx<+\infty.$$ 
Here also used the fact that 
$$
\begin{array}{rcl}
1 \leq \dfrac{2Np}{4+2\alpha-(N-2)(p-2)}\leq 2^*=\dfrac{2N}{N-2}.
\end{array}$$
The last estimates are verified due to the fact that $2_\alpha < p < 2_\alpha^* < 2^*$. 
This finishes the proof.
\end{proof}

\end{document}